\documentclass[12pt,leqno]{article}
\usepackage{amsmath}
\usepackage{subfigure}
\usepackage{amsthm}
\usepackage{amstext}
\usepackage{amsopn}
\usepackage{texdraw}
\usepackage{graphicx}
\usepackage{pdfpages}
\usepackage{multicol,graphicx,color}
\usepackage{pslatex}
\usepackage{amsthm}
\usepackage{amssymb}
\usepackage{latexsym}
\usepackage{lscape}
\usepackage{epsfig}
\usepackage{pstricks}
\usepackage{amsfonts}
\usepackage{enumerate}
\usepackage{exscale}
\usepackage{relsize}
\usepackage{multirow}
\usepackage{mathrsfs}
\usepackage{hyperref}
\usepackage{tikz}
\usetikzlibrary{shapes.geometric, arrows}
\usepackage{subcaption}

\newtheorem{theorem}{Theorem}[section]
\newtheorem{lemma}[theorem]{Lemma}

\newtheorem{corollary}[theorem]{Corollary}
\theoremstyle{definition}
\newtheorem{definition}[theorem]{Definition}

\theoremstyle{remark}
\newtheorem{remark}[theorem]{Remark}
\usepackage{tikz}
\usetikzlibrary{shapes.geometric, arrows}

\tikzstyle{startstop} = [rectangle, rounded corners, 
minimum width=5cm, 
minimum height=1cm,
text centered, 
text width=4.5cm, 
draw=black, 
fill=red!30]

\tikzstyle{io} = [trapezium, 
trapezium stretches=true, 
trapezium left angle=70, 
trapezium right angle=110, 
minimum width=5cm, 
minimum height=1cm, text centered, 
draw=black, fill=blue!30]

\tikzstyle{process} = [rectangle, 
minimum width=5cm, 
minimum height=1cm, 
text centered, 
text width=4.5cm, 
draw=black, 
fill=orange!30]

\tikzstyle{decision} = [diamond, 
minimum width=3cm, 
minimum height=1cm, 
text centered, 
text width=4.5cm, 
draw=black, 
fill=green!30]
\tikzstyle{arrow} = [thick,->,>=stealth]
\usepackage{multirow}

\begin{document}
\title{Degeneracy of Planar 4-body Kite Central Configurations: Restricted and Full Space Perspective}
\author{Shanzhong Sun\\
Department of Mathematics \\
and\\
Academy for Multidisciplinary Studies\\
Capital Normal University, Beijing 100048, P. R. China\\
sunsz@cnu.edu.cn\\
 Zhifu Xie\\
 School of Mathematics and Natural Science\\
 The University of Southern Mississippi\\
 Hattiesburg, MS 39406,USA\\
 Zhifu.Xie@usm.edu\\
 Peng You\\
  School of Statistics and Mathematics\\
  Hebei University of Economics and Business\\
  Shijiazhuang Hebei 050061, P. R. China\\
you-peng@163.com
  }
 \date{}
\maketitle

\begin{abstract} 
Based on our previously proposed direct approach to central configurations in the full configuration space, we study the degeneracy of planar four-body kite central configurations with arbitrary positive masses both in kite symmetry restricted subspace and in full configuration space.  We derive a factorization of the determinant of the Jacobian that substantially simplifies the subsequent analysis. The methods used to study the singularity of the Jacobian include rigorous interval arithmetic, the Krawczyk operator, a quantitative implicit function theorem, and computer-assisted proofs. Both convex and concave kite central configurations are considered. In the convex case, we prove the nondegeneracy in both the restricted subspace and the full configuration space. In the concave case, there is a unique smooth degeneracy curve in the restricted subspace; an additional degeneracy curve emerges in the full space setting, which is hidden in the subspace of kite configurations and makes it unavoidable to study the degeneracy in the full configuration space. We also establish several properties of these degeneracy curves. As an illustration to this complete characterization, we exam the well-known central configuration of equilateral triangle with equal masses at the vertices and a fourth mass at the center, which reveals additional features of degeneracy captured by our full-space framework.

\end{abstract}

\tableofcontents

\section{Introduction and Related Works}

The Newtonian $N$-body problem concerns the motion of $N$ point masses interacting through mutual gravitational attraction. Despite its simple formulation, the problem exhibits a rich variety of geometric and dynamical phenomena and remains one of the central topics in celestial mechanics. Among its most important special solutions are the homographic motions generated by the \emph{central configurations (CCs)}, where the acceleration vector of each mass is proportional to its position vector relative to the center of mass.

Let $q=(q_1,\ldots,q_N)\in \mathbb{R}^{2N}$ denote a planar configuration of $N$ bodies with masses $m_1,\ldots,m_N$. A configuration is called a central configuration if there exists a positive constant $\lambda$ such that
\begin{equation} \label{eq:cc_def1}
     \sum_{j \neq i} \frac{m_i m_j (q_j - q_i)}{|q_i - q_j|^3}=  -\lambda m_{i}(q_i - c)\quad i=1,\dots,N,
\end{equation}
where $c=c(q)=\frac{\sum_{1\leq i\leq N}m_iq_i}{\sum_{1\leq i\leq N}m_i}$ is the center of mass of configuration $q$. 
A central configuration can also be characterized as a critical point of the Newtonian potential function constrained to a level set of the moment of inertia. Let the Newtonian potential be
\begin{equation}\label{NU}
U(q)=\sum_{1\leq i<j\leq N} \frac{m_im_j}{\|q_i-q_j\|},
\end{equation}
and the moment of inertia be
\begin{equation}\label{NI}
    I(q)= \sum_{i=1}^N m_i\|q_i-c\|^2.
\end{equation}
Then the central configuration equation \eqref{eq:cc_def1} can be rewritten as
\begin{equation} \label{eq:cc_def2}
\nabla U(q)=-\lambda \nabla I(q).
\end{equation}
By virtue of the homogeneity of $U$ and $I$, the Lagrange multiplier $\lambda=\lambda(q)=U(q)/I(q)$ {when the configuration is central}. Central configurations play a fundamental role in the $N$-body problem because they generate homographic motions, describe the asymptotic behavior of solutions near total collision, and provide the relative equilibria of the system (see e.g., \cite{Moeckel14}, 2015).

The planar four-body central configuration problem occupies a particularly important position between the completely understood three-body case and the substantially more complicated $N(>4)$-body problems. During the last several decades, significant progress has been achieved in the classification of planar four-body central configurations. Special attention has been devoted to configurations possessing some symmetry, such as rhombi, isosceles trapezoids, and kites, because symmetry reduces the dimension of the problem while still revealing many essential geometric features. Kite central configurations are our main concern in the current paper, and more references will be given in due course.

One aspect that has received comparatively less attention is the occurrence of degenerate central configurations. A central configuration is said to be degenerate when the Jacobian matrix associated with the central configuration equations loses rank, or equivalently when the Hessian of the potential restricted to the normalized configuration manifold possesses a nontrivial kernel beyond the trivial directions induced by translation, rotation, and scaling. Degeneracy is closely related to bifurcation phenomena, changes in the number of central configurations, and the structure of the solution set in mass-parameter space. In many problems, degenerate configurations mark the transition between distinct branches of solutions and therefore play a crucial role in understanding the global geometry of the central configuration variety.

\subsection{Backgrounds on Degeneracy of Central Configurations}

The notion of degeneracy of central configurations dates back to the work of Palmore (\cite{Palmore1975}, 1975) who proved that the equilateral triangle configuration with equal masses at its vertices and a special fourth mass at the center of the triangle is a degenerate central configuration of the planar $4$-body problem. This was to answer a question raised by Smale (\cite{Smale1974}, 1974) which asked whether there exists a degenerate central configuration in the planar $N$-body problem for any $N\geq 4$. Then Palmore ( \cite{Palmore1976}, 1976) extended his example to the case of regular $N$-polygon with one more mass at the center of planar $(N+1)$-body problem and studied the measure of degenerate relative equilibria, 
{showing that non-degeneracy is generic}.  However he didn't consider whether the degeneracy gives rise to bifurcations. 

Sim\'o (\cite{Simo1978}, 1978) gave a complete numerical study of the central configurations with arbitrary masses in the planar  $4$-body problem with degeneracy and bifurcations discussed in details.
Later, Arenstorf(\cite{Arenstorf1982}, 1982) analytically investigated the degeneracy of the restricted $4$-body problem.
Meyer(\cite{Meyer1987}, 1987) considered the continuation problem of degenerate central configurations from the restricted to the full $4$-body problem.

Meyer and Schmidt (\cite{Meyer1988}, 1988) proved that degeneracy in the $(N+1)$-body problem does give rise to bifurcation for the cases $N=3$ and $N=4$. In order to remove the trivial zero eigenvalues of the Jacobian matrix of system \eqref{eq:cc_def1} --- which correspond to the continuous symmetries of translation, rotation, and scaling--- they adopted relative distances as the fundamental variables (Dziobek coordinates). More precisely, when $N=3$, they consider the augmented potential
\[
V(\lambda_0, \lambda_1, r_{12}, r_{13}, \dots, r_{34}) = U + \lambda_0 (I - {I_0}) + \frac{\lambda_1}{2} S,
\]
where
 $S$ is a function of relative distances that determines whether the configurations are coplanar. Then the equations defining central configurations are equivalent to
\[
\frac{\partial V}{\partial \lambda_0} = 0,\quad 
\frac{\partial V}{\partial \lambda_1} = 0,\quad 
\frac{\partial V}{\partial r_{ij}} = 0.
\]
Let $F = [F_1, F_2, \dots, F_8]=0$ denote the corresponding system. By continuing Palmore, Meyer and Schmidt examined degeneracy especially bifurcations for two specific configurations:
{(1)} The four-body problem with three equal masses sitting at the vertices of an equilateral triangle while a fourth mass at the center. {(2)} The five-body problem with four equal masses at the vertices of a square and a fifth at the center. {These are the {first} explicit examples showing the birth or death of additional families of central configurations stemming from degenerate ones.}


Barros and Leandro (\cite{BL11}, 2011 and \cite{BL14}, 2014) studied the degeneracy, bifurcations, and enumeration for the planar restricted $4$-body problem.

Later Rusu and Santoprete (\cite{Santoprete2015}, 2015) investigated bifurcations of central configurations in the four-body problem with some equal masses, more precisely the
case where three of the masses are equal and the case where there are two
pairs of equal masses which we will also address later. By using the relative distances as variables as did by Meyer and Schmidt (\cite{Meyer1988}), and applying interval arithmetic and the Krawczyk operator (\cite{Moore}) with equivariant bifurcation theory combined, they provided a rigorous computer-assisted proof of the existence of bifurcations. Their results showed that symmetric configurations can undergo bifurcations precisely when the Hessian of the potential restricted to the shape sphere becomes singular. Moreover, they identified specific mass parameters at which the kite configuration loses non-degeneracy.


\subsection{The Kite Family and Previous Results}

Rather than working in the full eight-dimensional configuration space for the planar $4$-body problem, several authors have investigated degeneracy within subspaces restricted by symmetries. One important example is the planar kite configuration, in which two masses lie on the axis of symmetry while the remaining two equal masses are positioned symmetrically with respect to that axis. Such central configurations could be convex or concave with dramatically different character as we will see later.

Bernat, Llibre, and P\'erez-Chavela (\cite{Bernat2009}, 2009) studied planar four-body central configurations with three equal masses and characterized the resulting kite central configurations and their symmetries, paving the way for the study of degeneracy in symmetric settings. Roberts (\cite{Roberts2025}, 2025) further investigated the existence, uniqueness, and linear stability of kite central configurations. By restricting the analysis to the symmetry-reduced space, he numerically identified a curve of degenerate concave central configurations. However, bifurcation curves connecting symmetric and asymmetric configurations could not be detected within this restricted framework.

More recently, Liu and Xie \cite{Liu2026} studied symmetric central configurations in the concave four-body problem with two pairs of equal masses. A one-parameter family of symmetric kite central configurations was analyzed within a subspace restricted by symmetry. There, the Jacobian generically has corank one, and only a single fold bifurcation occurs at a critical mass $m=m_0$. When the same family is viewed in the full planar configuration space, however, additional degeneracies emerge. In particular, at a different critical mass $\tilde m^{*}$ , a symmetric central configuration that appears nondegenerate in the restricted space becomes degenerate in the full space. This degeneracy gives rise to a pitchfork bifurcation and the emergence of asymmetric central configurations, as demonstrated in \cite{Santoprete2015}. As observed in Remark~2 of \cite{Liu2026}, this phenomenon is completely hidden by the imposed symmetry and therefore cannot be detected through a symmetry-reduced analysis. By contrast, the fold bifurcation at $m=m_0$ corresponds to a degeneracy that persists both in the restricted subspace and in the full configuration space, leading to the disappearance of symmetric solutions beyond the bifurcation point.

These examples illustrate that degeneracy is not merely a property of a particular symmetry class, but rather an intrinsic feature of the full configuration space. Restricting attention to symmetry-invariant subspaces may conceal important degeneracy mechanisms and thus reveal only a partial picture of the underlying bifurcation structure. This observation motivates the present work and highlights the necessity of studying degeneracy in the full planar configuration space.

A major challenge in such an approach is that the Jacobian matrix of the central configuration equations naturally possesses trivial zero eigenvalues arising from the invariance of the problem under translation, rotation, and scaling. These symmetries obscure the detection of genuine degeneracies and complicate the analysis of bifurcation phenomena. Previous studies have typically circumvented this difficulty by introducing configuration-dependent coordinates that reduce the number of variables and eliminate the trivial kernel. While effective in specific settings, such coordinate reductions are often tailored to particular symmetry classes and may not be suitable for investigating degeneracy in the full configuration space.

\subsection{Contributions of the Paper}

In this paper, we investigate the degeneracy of planar four-body kite central configurations in the full configuration space. Using the direct computational framework developed in \cite{Sun2025}, we derive explicit algebraic conditions characterizing degenerate kite central configurations. Through a combination of analytical methods and rigorous interval arithmetic computations, we identify the parameter values where degeneracy occurs and analyze the associated bifurcation structure.

Our main result shows that the degeneracy locus comprises two distinct curves in the mass-parameter space. One curve coincides with the degeneracy previously found by Roberts \cite{Roberts2025} in the symmetry reduced setting. The second curve emerges only in the full configuration space and was therefore invisible to symmetry-restricted analyses. These two curves are tangent to each other at the special point which is the well-known degenerate equilateral triangle configuration with one at its center. This newly discovered curve governs bifurcations between symmetric and asymmetric central configuration branches, demonstrating that symmetry reductions can mask essential degeneracy phenomena.

Putting together, these results provide a more complete characterization of degeneracy in $4$-body kite central configurations and enhance our understanding of the intricate relationships among symmetry, degeneracy and bifurcation in the planar $4$-body problem. More generally, they underscore the necessity of studying central configurations in the full configuration space, rather than exclusively within symmetry-invariant subspaces.

The remainder of the paper is organized as follows: In \S 2, we recall the direct computational framework of \cite{Sun2025} to study the degeneracy of central configurations. \S 3 is about the settings for the kite configurations and the meaning of degeneracy in the restricted and full space framework, preparing the background for later sections. We derive an interesting factorization of the Jacobian (Theorem \ref{Prop:J2Entries}), and the factors correspond precisely the bifurcations along the vertical and horizontal directions.
The following two sections constitute the main body of the paper.
In \S 4, we establish the nondegeneracy of convex kite central configurations in restricted subspace (Theorem \ref{nondegrestricted}) and in full space (Theorem \ref{nondegfull}) by interval algorithms.
In \S 5, we combine the interval arithmetic, the Krawczyk operator, and the quantitative implicit function theorem to study the degenerate concave kite central configurations. We first show the existence of the unique smooth degeneracy curve in the restricted case (Theorem \ref{thm:DegReduced}). In the full space, we prove that there are two degeneracy curves which we call horizontal and vertical degeneracy curve (Theorem \ref{thm:Fh} and Theorem \ref{thm:Fv}), and some interesting properties are established. The central configuration of equilateral triangle with a fourth mass at the center play a prominent role as the only intersection point of the vertical and horizontal degeneracy curves, and all interesting properties emerge there. We also show the behaviors of the masses along the degeneracy loci in Figure 6, Figure 7 and Figure 8. 
In the Appendix we provide the proof of a technical lemma used in the proof of Theorem \ref{Prop:J2Entries}.

{Once the degeneracy of central configurations are clear, the natural next step is the bifurcation analysis which we leave to the future.}
\section{Degeneracy of Central Configurations: a Direct and Full Space Approach}\label{secDeg}
 Using the notation in \cite{Sun2025}, let $q_j \in \mathbf{R}^2$ be the position of the $j^{th}$ body in plane with mass $m_j$. Let $F(q,m)=[F_i(q,m)]$ be the function of $\mathbf{R}^{2N}\times \mathbf{R}^{N}\rightarrow \mathbf{R}^{2N}$ defined by
\begin{equation}\label{CF3}
F_i(q,m)=\sum_{j=1,j\not=i}^{N}
\frac{m_j(q_j-q_i) }{\|q_j-q_i\|^3}+\frac{U}{I}
(q_i-c)  \hspace{1cm} 1\leq i\leq N,
\end{equation}
where the potential $U$ and the moment of inertia $I$ are given by (\ref{NU}) and (\ref{NI}).
The configuration $q_0$ is a central configuration for $m_0$ which corresponds  to $F(q_0,m_0)=0$. Then the translation, rotation, and scaling of $q_0$  are also central configurations.

Due to invariants of central configurations under translation, rotation, and scaling, there are trivial zero eigenvalues of its Jacobian matrix. The paper \cite{Sun2025} presented four different equivalent forms to study the degeneracy of central configurations by systematically removing trivial zero eigenvalues. Here is the Form IV: four trivial zero eigenvalues for planar central configurations. 

\begin{lemma}[\textbf{Proposition 3.12} in \cite{Sun2025}]\label{ThmForm3}
    Let $q_0$ be a planar central configuration of mass $m_0$ defined by equation \eqref{CF3} with $i^{th}$ position ${q_0}_i=[x_i, y_i].$ Let $$P =\left[\begin{array}{cc}
B_1 & 0 
\\
 B_2 & I
\end{array}\right]$$ be an invertible matrix with $[B_1,B_2]^T$ constructed by the four zero eigenvectors corresponding to translation, rotation, and scaling. A typical $P$ looks like
\[
P= \left[\begin{array}{cccccccc}
1 & 0 & -y_1 & x_1 & 0 & 0 & 0 & 0 \\
0 & 1 & x_1  & y_1 & 0 & 0 & 0 & 0 \\
1 & 0 & -y_2 & x_2 & 0 & 0 & 0 & 0 \\
0 & 1 & x_2  & y_2 & 0 & 0 & 0 & 0 \\
\cdots & & & \cdots & & & \cdots &\\
\cdots & & & \cdots & & & \cdots &\\
1 & 0 & -y_N & x_N & 0 & 0 & 1 & 0 \\
0 & 1 & x_N  & y_N & 0 & 0 & 0 & 1
\end{array}\right].
\]
  Then  $P^{-1}Jac(F)|_{(q_0,m_0)} P$ has the form$ \left[\begin{array}{cc}
0 & J_1
\\
0 & J_2
\end{array}\right]$ with $J_2$ a $(2N-4)\times (2N-4) $ matrix. 
\end{lemma}

\begin{definition}[\textbf{Nondegeneracy of a Central Configuration} (Form IV in \cite{Sun2025})]
    A central configuration \( q_0 \) for mass \( m_0 \) is said to be nondegenerate if \( \det J_2 \neq 0 \) in Lemma \ref{ThmForm3}. Otherwise, it is considered to be degenerate.
\end{definition}
\begin{theorem}\label{Thm:DegForm}
    Assume Lemma \ref{ThmForm3} holds. Let $Jac(F)|_{(q_0,m_0)} =\left[\begin{array}{cc}
J_{11} & J_{12} 
\\
 J_{21} & J_{22}
\end{array}\right]$. Then $J_2= J_{22}-B_2B_1^{-1}J_{12}$.
\end{theorem}
\begin{proof}
  It is easy to get $P^{-1}=\left[\begin{array}{cc}
B_1^{-1} & 0
\\
-B_2B_1^{-1} & I
\end{array}\right]$ by the definition of $P$.  By direct computation, we can prove that $J_2$ has the given form. 
\end{proof}
A key consequence of this theorem is that degeneracy in the full planar four-body problem can be detected by analyzing the  $(4\times4)$  matrix  $(J_2$), rather than the full  $(8\times8)$  Jacobian considered in previous studies \cite{Meyer1988,Santoprete2015}. Thus, the trivial degeneracies arising from translation, rotation, and scaling are eliminated automatically, leaving a compact matrix that captures all essential degeneracy information. We now employ this framework to investigate the degeneracy of kite central configurations.

\section{Kite Central Configurations: Settings and Degeneracy}
 Without loss of generality, assume that $m_1=m_2 =1$ and
the four bodies of the kite have positions (see figure \ref{fig1})
\begin{equation}\label{KiteCon}
q_1=(-1,0), q_2=(1,0), q_3=(0, a), \hbox{ and }  q_4=(0, b).
\end{equation}
 Unless otherwise stated, we assume that $m_3$ is below $m_4$, i.e. $a<b$ and $m_4$ is above the $x$-axis, i.e. $b>0.$ There are two types of configurations: convex when $a<0<b$ (see right of Figure \ref{fig1}), and concave when $0<a<b$ (see left of Figure \ref{fig1}). 

 \begin{figure}[!ht]
\centering
\begin{minipage}[b]{0.48\textwidth}
    \centering
    \includegraphics[width=\textwidth]{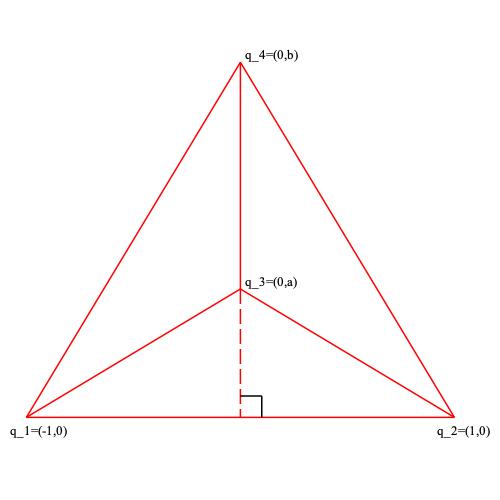}
\end{minipage}
\hfill
\begin{minipage}[b]{0.48\textwidth}
    \centering
    \includegraphics[width=\textwidth]{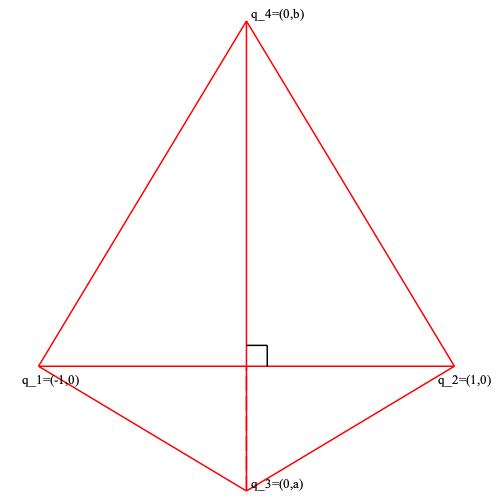}
\end{minipage}
\caption{\label{fig1}  Two kite configurations: The concave configuration (Left) has $0<a<b$ with $m_3$ inside the triangle formed by $m_1$, $m_2$ and $m_4$. The convex configuration (Right) has $a<0<b.$  }
\end{figure}

By the symmetry of the kite configurations, 
the eight equations (\ref{CF3}) of CC are reduced to two equations as follows:
\begin{equation}\label{eqR}
 \left\{\begin{array}{l}    \frac{b \left((b^{2}+1)^{\frac{3}{2}} -8\right)}{4 \left(b^{2}+1\right)^{\frac{3}{2}}}+ \frac{(b-a)^3-(a^2+1)^{\frac{3}{2}}}{(b-a)^2(a^2+1)^{\frac{3}{2}}}m_3=0,\\
  \frac{a \left((a^{2}+1)^{\frac{3}{2}}-8\right)}{4 \left(a^{2}+1\right)^{\frac{3}{2}}}+\frac{\left(b^{2}+1\right)^{\frac{3}{2}}-\left(b -a \right)^{3}}{\left(b^{2}+1\right)^{\frac{3}{2}} \left(b -a \right)^{2}} m_4=0.\\
   \end{array}\right.
\end{equation}
Let \begin{equation}\label{eqnm3}
   n_{m_3}=b \left(a^{2}+1\right)^{\frac{3}{2}} \left(a -b \right)^{2}\left(\left(b^{2}+1\right)^{\frac{3}{2}}-8\right), 
   \end{equation}
    \begin{equation}\label{eqdm3}
    d_{m_3}= 4 \left(b^{2}+1\right)^{\frac{3}{2}} \left(\left(a^{2}+1\right)^{\frac{3}{2}}+\left(a -b \right)^{3}\right),
    \end{equation}
    \begin{equation}\label{eqnm4}
       n_{m_4}= \left(b^{2}+1\right)^{\frac{3}{2}} a \left(a -b \right)^{2}\left(\left(a^{2}+1\right)^{\frac{3}{2}}-8\right), 
    \end{equation}
   \begin{equation}\label{eqdm4}
        d_{m_4}=-4 \left(a^{2}+1\right)^{\frac{3}{2}} \left(\left(b^{2}+1\right)^{\frac{3}{2}}+\left(a -b \right)^{3}\right).
   \end{equation}
   Since $b>a,$ equations \eqref{eqR} are equivalent to 
   \begin{equation}\label{eqR2}
 \left\{\begin{array}{l} 
   g_1\equiv n_{m_3}-d_{m_3}m_3=0,\\
   g_2\equiv n_{m_4}-d_{m_4}m_4=0. 
     \end{array}\right.
\end{equation}
The following theorem is obvious. 
\begin{theorem}
    Assume that the kite shape configuration $q$ is normalized as \eqref{KiteCon}, and that $d_{m_3}\not=0$ and $d_{m_4}\not=0$. Then $q$ is a central configuration corresponding to mass vector $m=[1,1, m_3, m_4]$ if and only if $m_3$ and $m_4$ are given explicitly by:
    \begin{equation}\label{eqm34}
        m_3=\frac{n_{m_3}}{d_{m_3}}, \hbox{ and } m_4= \frac{n_{m_4}}{d_{m_4}}.
    \end{equation}
\end{theorem}
\subsection{Degeneracy in Restricted Kite Symmetrical Subspace}
The degeneracy of kite central configurations has been studied recently by Bernat,  Llibre, and  P\'erez-Chavela (\cite{Bernat2009}), Roberts (\cite{Roberts2025}), and Liu and Xie (\cite{Liu2026, Liu2026_2}) in the restricted symmetrical subspace under some extra constraints. We will first investigate the degeneracy of the reduced equations \eqref{eqR2} with respect to position variables $x=(a,b)$ and mass parameters $\mu=(m_3, m_4),$ i.e. $$G(x,\mu)=[g_1,g_2]^T=[0,0]^T.$$
The Jacobian of $G$ is a $2\times 2$ matrix
$$Jac(G)=\left[\begin{array}{cc}
\frac{\partial n_{m_3}}{\partial a} -\frac{\partial d_{m_3}}{\partial a}m_3 & \frac{\partial n_{m_3}}{\partial b} -\frac{\partial d_{m_3}}{\partial b}m_3
\\
\frac{\partial n_{m_4}}{\partial a} -\frac{\partial d_{m_4}}{\partial a}m_4 & \frac{\partial n_{m_4}}{\partial b} -\frac{\partial d_{m_4}}{\partial b}m_4
\end{array}\right].$$
At a kite central configuration, by using \eqref{eqm34}, we have 
$$Jac(G)=\left[\begin{array}{cc}
\frac{1}{d_{m_3}}\left(\frac{\partial n_{m_3}}{\partial a}d_{m_3} -\frac{\partial d_{m_3}}{\partial a}n_{m_3}\right) & \frac{1}{d_{m_3}}\left(\frac{\partial n_{m_3}}{\partial b}d_{m_3} -\frac{\partial d_{m_3}}{\partial b}n_{m_3}\right)
\\
\frac{1}{d_{m_4}}\left(\frac{\partial n_{m_4}}{\partial a}d_{m_4} -\frac{\partial d_{m_4}}{\partial a}n_{m_4}\right) & \frac{1}{d_{m_4}}\left(\frac{\partial n_{m_4}}{\partial b}d_{m_4} -\frac{\partial d_{m_4}}{\partial b}n_{m_4}\right)
\end{array}\right].$$
The degeneracy of the kite central configurations is characterized by the equation $\det(Jac(G)) $ $=0$. To simplify the computations, let $F_{red}(a,b)=d_{m_3}d_{m_4} \det\left( Jac(G)\right)$, that is
\begin{equation}\label{Dab_reduced}
  F_{red}(a,b)= \left(\frac{\partial n_{m_3}}{\partial a}d_{m_3} -\frac{\partial d_{m_3}}{\partial a}n_{m_3}\right)\left(\frac{\partial n_{m_4}}{\partial b}d_{m_4} -\frac{\partial d_{m_4}}{\partial b}n_{m_4}\right)-  
\end{equation}
$$\left(\frac{\partial n_{m_4}}{\partial a}d_{m_4} -\frac{\partial d_{m_4}}{\partial a}n_{m_4}\right)\left(\frac{\partial n_{m_3}}{\partial b}d_{m_3} -\frac{\partial d_{m_3}}{\partial b}n_{m_3}\right).
$$
When $d_{m_3}\not=0$ and $d_{m_4}\not=0$,  $F_{red}(a,b)=0$ if and only if $\det(Jac(G))=0$. 

This expression will be analyzed further for both convex and concave configurations to determine the degeneracy of kite central configurations. We should stress that this degeneracy, however, exists only on the restricted symmetric subspace. As observed in \cite{Liu2026}, this restriction does not capture certain points of degeneracy in the whole space, motivating our subsequent analysis in the full space.

\subsection{Degeneracy in Full Configuration Space}
Before stating the degeneracy theorem in full kite space, we present two lemmas that will be used in its proof. \begin{lemma}\label{lem:kernel}
Let
\[
\Phi:\mathbb{R}^2\setminus\{0\}\longrightarrow\mathbb{R}^2,
\qquad
\Phi(z)=\frac{z}{\|z\|^3},
\]
where $z=(z_1,z_2)^T$ and $r=\|z\|=\sqrt{z_1^2+z_2^2}$.
Then the Jacobian matrix of $\Phi$ is
\[
D\Phi(z) = \frac1{r^3}I_2 - \frac3{r^5}zz^T,
\]
or equivalently,
\[
D\Phi(z) =
\begin{pmatrix}
\dfrac1{r^3}-\dfrac{3z_1^2}{r^5}
&
-\dfrac{3z_1z_2}{r^5}
\\[3mm]
-\dfrac{3z_1z_2}{r^5}
&
\dfrac1{r^3}-\dfrac{3z_2^2}{r^5}
\end{pmatrix}.
\]
\end{lemma}
\begin{proof}
    This follows by a straightforward calculation, which we omit for brevity.
\end{proof}
\begin{lemma}\label{lem:cancellation-normalization}
Let
\[
\mathcal{N}(q)=
\begin{pmatrix}
\lambda(q)(q_1-c(q))\\
\lambda(q)(q_2-c(q))\\
\lambda(q)(q_3-c(q))\\
\lambda(q)(q_4-c(q))
\end{pmatrix},
\qquad
\lambda(q)=\frac{U(q)}{I(q)},
\]
and partition its Jacobian according to
 $(x_1,y_1,x_2,y_2)$ and $(x_3,y_3,x_4,y_4):$ 
\[
D\mathcal{N}(q) =
\begin{pmatrix}
N_{11}&N_{12}\\
N_{21}&N_{22}
\end{pmatrix}.
\]
At the kite configuration
$ q_1=(-1,0), q_2=(1,0), q_3=(0,a), q_4=(0,b), $
{with \(m_1=m_2=1\),} one has
\[
N_{22}-B_2B_1^{-1}N_{12} = \lambda I_4.
\]
In particular, the contributions involving \(Dc\) and \(D\lambda\)
cancel identically in the Jacobian.
\end{lemma}

Applying  Theorem \ref{Thm:DegForm} to kite central configurations, one obtains the following theorem. 
\begin{theorem}\label{Prop:J2Entries}
Let
\[
q_1=(-1,0),\qquad
q_2=(1,0),\qquad
q_3=(0,a),\qquad
q_4=(0,b),
\]
where $a<b$ and $b>0$, and assume $m_1=m_2=1$.
Let
\[
r_a=\sqrt{1+a^2},\qquad
r_b=\sqrt{1+b^2},\qquad
d=b-a,
\]
and
\[
\lambda_0=\frac{U_0}{I_0},
\]
where
\[
U_0=\frac12+\frac{2m_3}{r_a}
+\frac{2m_4}{r_b}
+\frac{m_3m_4}{d},
\]
and
\[
I_0=
2+m_3a^2+m_4b^2
-\frac{(m_3a+m_4b)^2}
{2+m_3+m_4}.
\]

Then the Jacobian matrix
\[
J_2=J_{22}-B_2B_1^{-1}J_{12}
\]
takes the form
\[
J_2=
\begin{pmatrix}
a_{11}&0&a_{13}&0\\
0&a_{22}&0&a_{24}\\
a_{31}&0&a_{33}&0\\
0&a_{42}&0&a_{44}
\end{pmatrix},
\]
where
\[
a_{11}
=
\lambda_0
+\frac{2m_3}{r_a^3}
+\frac{2(2-a^2)}{r_a^5}
-\frac{m_4}{d^3},\qquad
a_{13}
=
m_4\left(
\frac{1}{d^3}
+\frac{2}{r_b^3}
-\frac{3bd}{r_b^5}
\right),
\]
\[
a_{31}
=
m_3\left(
\frac{1}{d^3}
+\frac{2}{r_a^3}
+\frac{3ad}{r_a^5}
\right),\qquad
a_{33}
=
\lambda_0
+\frac{2m_4}{r_b^3}
+\frac{2(2-b^2)}{r_b^5}
-\frac{m_3}{d^3},
\]
\[
a_{22}
=
\lambda_0
-\frac{m_3}{r_a^3}
+\frac{2(2a^2-1)}{r_a^5}
+\frac{2m_4}{d^3},\qquad
a_{24}
=
-m_4\left(
\frac{2}{d^3}
+\frac{1}{r_b^3}
-\frac{3bd}{r_b^5}
\right),
\]
\[
a_{42}
=
-m_3\left(
\frac{2}{d^3}
+\frac{1}{r_a^3}
+\frac{3ad}{r_a^5}
\right),\qquad
a_{44}
=
\lambda_0
-\frac{m_4}{r_b^3}
+\frac{2(2b^2-1)}{r_b^5}
+\frac{2m_3}{d^3}.
\]
Consequently,
\begin{equation}\label{detJ2}
\det(J_2)=
\left(a_{11}a_{33}-a_{13}a_{31}\right)
\left(a_{22}a_{44}-a_{24}a_{42}\right).
\end{equation}
\end{theorem}

\begin{proof}
Let
\[
\Phi:\mathbb{R}^2\setminus\{0\}\longrightarrow\mathbb{R}^2,
\qquad
\Phi(z)=\frac{z}{\|z\|^3}.
\]
By Lemma~\ref{lem:kernel},
\[
D\Phi(z)
=
\frac{1}{\|z\|^3}I_2
-
\frac{3}{\|z\|^5}zz^T.
\]
For \(i\neq j\), define
\[
H_{ij}=D\Phi(q_j-q_i).
\]
Since
{
\[
D\Phi(-z)=D\Phi(z),
\]
we also have
\[
H_{ij}=H_{ji}.
\]
}
The interaction of the \(j\)-th body with the \(i\)-th body is
\[
m_j\Phi(q_j-q_i).
\]
Differentiating with respect to the two position variables involved
gives
\begin{equation}\label{eq:pair-derivatives}
D_{q_j}\!\left[m_j\Phi(q_j-q_i)\right]
=
m_jH_{ij},
\qquad
D_{q_i}\!\left[m_j\Phi(q_j-q_i)\right]
=
-m_jH_{ij}.
\end{equation}
Thus, the derivative with respect to the position of the other body
has a positive sign, whereas the derivative with respect to the
position of the body on which the force acts has a negative sign. Consider the kite configuration
\[
q_1=(-1,0),\qquad
q_2=(1,0),\qquad
q_3=(0,a),\qquad
q_4=(0,b),
\]
where $m_1=m_2=1, a<b.$ 
For later use, introduce the matrices
\[
H_a^+=D\Phi(1,a),
\qquad
H_a^-=D\Phi(-1,a),\qquad
H_b^+=D\Phi(1,b),
\qquad
H_b^-=D\Phi(-1,b),
\]
and
\[
H_d=D\Phi(0,d).
\]
By Lemma~\ref{lem:kernel},
\[
H_a^+
=
\frac{1}{r_a^5}
\begin{pmatrix}
a^2-2&-3a\\
-3a&1-2a^2
\end{pmatrix},
\qquad
H_a^-
=
\frac{1}{r_a^5}
\begin{pmatrix}
a^2-2&3a\\
3a&1-2a^2
\end{pmatrix},
\]
and similarly,
\[
H_b^+
=
\frac{1}{r_b^5}
\begin{pmatrix}
b^2-2&-3b\\
-3b&1-2b^2
\end{pmatrix},
\qquad
H_b^-
=
\frac{1}{r_b^5}
\begin{pmatrix}
b^2-2&3b\\
3b&1-2b^2
\end{pmatrix}.
\]
Moreover,
\[
H_d
=
\frac{1}{d^3}
\begin{pmatrix}
1&0\\
0&-2
\end{pmatrix}.
\]

Write the vector field as
\[
F=\widehat F+\mathcal N,
\]
where
\[
\widehat F_i
=
\sum_{j\neq i}m_j\Phi(q_j-q_i)
\]
is the interaction part and
\[
\mathcal N_i
=
\lambda(q)(q_i-c(q)),
\qquad
\lambda(q)=\frac{U(q)}{I(q)},
\]
is the normalization part. Partition the corresponding Jacobians
according to the variables
\[
(q_1,q_2)
\quad\text{and}\quad
(q_3,q_4).
\]
Thus,
\[
D\widehat F
=
\begin{pmatrix}
\widehat J_{11}&\widehat J_{12}\\
\widehat J_{21}&\widehat J_{22}
\end{pmatrix} \hbox{ and }
D\mathcal{N}(q)
=
\begin{pmatrix}
N_{11}&N_{12}\\
N_{21}&N_{22}
\end{pmatrix}.
\]

Using \eqref{eq:pair-derivatives}, the block containing the derivatives
of the first two equations with respect to \(q_3\) and \(q_4\) is
\begin{equation}\label{eq:J12-interaction}
\widehat J_{12}
=
\begin{pmatrix}
m_3H_a^+&m_4H_b^+\\
m_3H_a^-&m_4H_b^-
\end{pmatrix}.
\end{equation}

Let
\[
C=B_2B_1^{-1}
=
\begin{pmatrix}
\frac12&\frac a2&\frac12&-\frac a2\\
-\frac a2&\frac12&\frac a2&\frac12\\
\frac12&\frac b2&\frac12&-\frac b2\\
-\frac b2&\frac 12&\frac b2&\frac12
\end{pmatrix}.
\]
To make the reduction explicit, direct matrix multiplication gives

\begin{equation}\label{eq:CJ12-simplified}
C\widehat J_{12}
=
\begin{pmatrix}
-\dfrac{2m_3}{r_a^3}
&0
&-m_4\left(
\dfrac{2}{r_b^3}
-\dfrac{3bd}{r_b^5}
\right)
&0
\\[3mm]
0
&\dfrac{m_3}{r_a^3}
&0
&m_4\left(
\dfrac{1}{r_b^3}
-\dfrac{3bd}{r_b^5}
\right)
\\[3mm]
-m_3\left(
\dfrac{2}{r_a^3}
+\dfrac{3ad}{r_a^5}
\right)
&0
&-\dfrac{2m_4}{r_b^3}
&0
\\[3mm]
0
&m_3\left(
\dfrac{1}{r_a^3}
+\dfrac{3ad}{r_a^5}
\right)
&0
&\dfrac{m_4}{r_b^3}
\end{pmatrix}.
\end{equation}

In particular, \(C\widehat J_{12}\) has the form
\[
C\widehat J_{12}
=
\begin{pmatrix}
*&0&*&0\\
0&*&0&*\\
*&0&*&0\\
0&*&0&*
\end{pmatrix},
\]
which already exhibits the decoupling of the horizontal and vertical
perturbations.

For the lower-right block {$\widehat J_{22}$}, the diagonal {subblocks} are self-derivatives
and therefore carry the negative sign in
\eqref{eq:pair-derivatives}. The off-diagonal blocks are derivatives
with respect to the other bodies and carry the positive sign. Hence
\begin{equation}\label{eq:J22-interaction}
\widehat J_{22}
=
\begin{pmatrix}
-\bigl(H_a^++H_a^-+m_4H_d\bigr)&m_4H_d\\[1mm]
m_3H_d&-\bigl(H_b^++H_b^-+m_3H_d\bigr)
\end{pmatrix}
\end{equation}
$$= \begin{pmatrix} \dfrac{2(2-a^2)}{r_a^5}-\dfrac{m_4}{d^3} & 0 & \dfrac{m_4}{d^3} & 0 \\[3mm] 0 & \dfrac{2(2a^2-1)}{r_a^5}+\dfrac{2m_4}{d^3} & 0 & -\dfrac{2m_4}{d^3} \\[3mm] \dfrac{m_3}{d^3} & 0 & \dfrac{2(2-b^2)}{r_b^5}-\dfrac{m_3}{d^3} & 0 \\[3mm] 0 & -\dfrac{2m_3}{d^3} & 0 & \dfrac{2(2b^2-1)}{r_b^5}+\dfrac{2m_3}{d^3} \end{pmatrix}. $$

By Lemma~\ref{lem:cancellation-normalization}, the normalization
part satisfies
\[
N_{22}-CN_{12}=\lambda_0 I_4.
\]
Consequently,
\begin{equation}\label{eq:J2-interaction-decomposition}
J_2
=
\lambda_0 I_4
+
\widehat J_{22}
-
C\widehat J_{12}.
\end{equation}

Substituting \eqref{eq:CJ12-simplified} and
\eqref{eq:J22-interaction} into
\eqref{eq:J2-interaction-decomposition} and carrying out the matrix
multiplication, we obtain
\[
J_2=
\begin{pmatrix}
a_{11}&0&a_{13}&0\\
0&a_{22}&0&a_{24}\\
a_{31}&0&a_{33}&0\\
0&a_{42}&0&a_{44}
\end{pmatrix}.
\]
The vanishing mixed entries follow from the cancellation of the
off-diagonal entries of \(H_a^+\) with those of \(H_a^-\), and
similarly for \(H_b^+\) and \(H_b^-\). \(a_{ij}\)s are given as in the theorem.

Therefore,
\[\det J_2=
\left(a_{11}a_{33}-a_{13}a_{31}\right)
\left(a_{22}a_{44}-a_{24}a_{42}\right).
\]
\end{proof}

As a result of Theorem \ref{ThmForm3} and Theorem \ref{Prop:J2Entries}, we have
\begin{theorem}\label{Deg:J2}
    Assume that the kite shape $q$ satisfying equation \eqref{KiteCon} is a central configuration for $m=[1,1, m_3, m_4]$, where $m_3$ and $m_4$ satisfy the equation \eqref{eqR2}. Then the central configuration is degenerate if $\det(J_2)$ given by \eqref{detJ2} is zero. 
\end{theorem}

\begin{remark}
 If the coordinates are reordered from
\[
(x_3,y_3,x_4,y_4)
\]
to
\[
(x_3,x_4,y_3,y_4),
\]
then $J_2$ becomes block diagonal:
\[
\begin{pmatrix}
a_{11}&a_{13}&0&0\\
a_{31}&a_{33}&0&0\\
0&0&a_{22}&a_{24}\\
0&0&a_{42}&a_{44}
\end{pmatrix}.
\]
The upper-left block represents the linearization with respect to horizontal ($x$-direction) perturbations and is associated with non-symmetric bifurcations. The lower-right block represents the linearization with respect to vertical ($y$-direction) perturbations and is associated with symmetric bifurcations. 

Since interval arithmetic is not well suited for expressions with denominators that may become small near the boundary of the admissible region, we first reformulate the two factors of $\det(J_2)$  into equivalent polynomial-like form before attempting to locate its zero set. Using notations $r_a=\sqrt{a^2+1}, r_b=\sqrt{b^2+1}$ and $d=b-a$,  equations \eqref{eqnm3}, \eqref{eqdm3}, \eqref{eqnm4} and \eqref{eqdm4} for $nm_3$, $dm_3$, $nm_4$ and $dm_4$ can be written as
$$nm_3=b r_a^{3} d^{2} \left(r_b^{3}-8\right),\quad dm_3= 4 r_b^{3} \left(r_a^{3}-d^{3}\right),$$
$$ nm_4=ar_b^{3} d^{2} \left(r_a^{3}-8\right),\quad dm_4= -4 r_a^{3} \left(r_b^3-d^{3}\right).$$

First note that $\lambda _0$ is also equal to $$\lambda_0= \frac{1}{4}+\frac{m_3}{r_a^3}+\frac{m_4}{r_b^3}.$$
Define the denominator-clearing factors $$C_{factorH}=16 d^{2} \left(d^{3}-r_b^{3}\right) \left(d^{3}-r_a^{3}\right) r_a^{8} r_b^{8},
$$
and $$C_{factorV}=16 d^{2} \left(d^{3}-r_b^{3}\right) \left(d^{3}-r_a^{3}\right) r_a^{5} r_b^{5}.$$
After substituting the expressions for $m_3, m_4$, and $\lambda_0$, we obtain
$$\hat{a}_{11}=-\left(d^{3}-r_a^{3}\right) r_b^{3} \left(a \left(-r_a^{5}+8 a d +8 r_a^{2}\right)-d \left(r_a^{5}+16\right)\right)-3 b \,d^{3} r_a^{5} \left(r_b^{3}-8\right),$$
$$\hat{a}_{13}=-r_a^{3} a \left(r_a^{3}-8\right) \left(3 b \,d^{4}-2 d^{3} r_b^{2}-r_b^{5}\right),\hspace{0.3cm} \hat{a}_{31}=-r_b^{3} b \left(r_b^{3}-8\right) \left(3 a \,d^{4}+2 d^{3} r_a^{2}+r_a^{5}\right),$$
and 
$$\hat{a}_{33}=\left(d^{3}-r_b^{3}\right) r_a^{3} \left(b \left(-r_b^{5}-8 b d +8 r_b^{2}\right)+d \left(r_b^{5}+16\right)\right)+3 a \,d^{3} r_b^{5} \left(r_a^{3}-8\right).$$
Similarly,
$$\hat{a}_{22}=a \left(d^{3}+2 r_b^{3}\right)r_a^{2} \left(r_a^{3}-8\right)+d \left(d^{3}-r_b^{3}\right) \left(r_a^{5}+16 a^{2}-8\right),$$
$$\hat{a}_{24}=a \left(r_a^{3}-8\right) \left(3 b \,d^{4}-d^{3} r_b^{2}-2 r_b^{5}\right), \hspace{0.3cm} \hat{a}_{42}=b \left(r_b^{3}-8\right) \left(3 a \,d^{4}+d^{3} r_a^{2}+2 r_a^{5}\right),$$
and 
$$\hat{a}_{44}=-b \left(d^{3}+2 r_a^{3}\right) r_b^{2} \left(r_b^{3}-8\right)+d \left(d^{3}-r_a^{3}\right) \left(r_b^{5}+16 b^{2}-8\right).$$
Therefore \begin{equation}\label{Fj2h}
    F_{J_{2h}}(a,b)=(a_{11}a_{33}-a_{31}a_{13})C_{factorH}=\hat{a}_{11}\hat{a}_{33}-\hat{a}_{31}\hat{a}_{13}
\end{equation}
  and
  \begin{equation}\label{Fj2v}
    F_{J_{2v}}(a,b)=(a_{22}a_{44}-a_{42}a_{24})C_{factorV}=\hat{a}_{22}\hat{a}_{44}-\hat{a}_{24}\hat{a}_{42}.
\end{equation}
These expressions contain no potentially small denominators and hence are better suited
for rigorous interval evaluations. Since the denominator-clearing factors do not vanish in the
admissible region, the condition $\det(J_2)=0$ is equivalent to $F_{J_{2h}}F_{J_{2v}}=0.$

\end{remark}


\section{Convex Kite Central Configurations}

In this section, we investigate the convex kite central configurations with convention \eqref{KiteCon} and $a<0<b.$ By the positivity of $m_3$ and $m_4$, we have the admissible region for convex kite central configurations. 

\begin{definition}
 Let $\mathcal{R}_X$ denote the open admissible region bounded by the vertical line 
$a = -\sqrt{3}$ and horizontal line $b = \sqrt{3}$, together with the two curves
\[
b = a + \sqrt{a^{2}+1} \quad\text{(blue curve)}, \qquad
a = b - \sqrt{b^{2}+1} \quad\text{(red curve)},
\]
as illustrated in the left panel of Figure~\ref{fig2}. Or put it equivalently,
$$\mathcal{R}_X=\left\{(a,b)| -\sqrt{3} <a< \sqrt{3}- 2,  a + \sqrt{a^{2}+1}<b<\sqrt{3} \right\} $$ $$ \bigcap \left\{ (a,b) |2-\sqrt{3}<b<\sqrt{3}, -\sqrt{3}<a<b - \sqrt{b^{2}+1} \right\}.  $$
The intersection points of these boundary curves are $A_1=[-\sqrt{3},  2 - \sqrt{3}],$  $B_1=[-\frac{\sqrt{3}}{3},$ $ \frac{\sqrt{3}}{3}]$,
     $C_1=[\sqrt{3}-2, \sqrt{3}]$, and $D_1=[-\sqrt{3},\sqrt{3}].$
\end{definition}

\begin{figure}[ht]
    \centering
    \includegraphics[width=0.6\textwidth,height=3in]{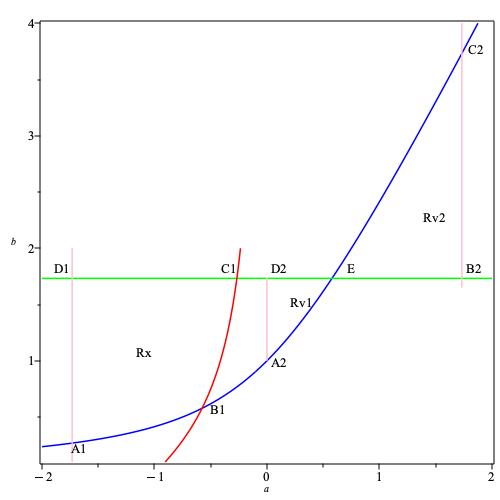}
    \caption{\label{fig2}  Admissible Regions: The admissible region for convex kite CC (Left $\mathcal{R}_{X}$) is bounded by four curves in second quadrant (red curve: $d_{m_4} =0$; blue curve: $d_{m_3}=0$; pink: $n_{m_4}=0$; green: $n_{m_3}=0$). The admissible region for concave kite CC (Right: $\mathcal{R}_{v1}$ and $\mathcal{R}_{v2}$ and special point $E$) is bounded by four curves in first quadrant.  }
\end{figure}

\begin{theorem}
    A kite configuration $q$ satisfying \eqref{KiteCon} is a convex kite central configuration for masses $m = [1,1,m_3,m_4]$ with $m_3>0$ and $m_4>0$ if and only if the point $(a,b)$ lies in the open admissible region $\mathcal{R}_X$. 
\end{theorem}

\begin{proof}
   A kite configuration $q$ satisfying \eqref{KiteCon} is a convex kite central configuration for masses $m = [1,1,m_3,m_4]$ if and only if $m_3$ and $m_4$ are given by equation \eqref{eqm34} and are positive. Notice that $a<0<b$ for convex configuration. It is easy to know that from equation \eqref{eqnm3}, the numerator of $m_3$ is positive for $b>\sqrt{3}$ and it is negative for $0<b< \sqrt{3}$. From equation \eqref{eqdm3}, the denominator of $m_3$ is positive for $b<a+\sqrt{a^2+1}$ and it is negative for $b>a+\sqrt{a^2+1}$. Thus  in order to have $m_3>0$,  $a+\sqrt{a^2+1}<b<\sqrt{3}$. \\
   On the other hand, from equation \eqref{eqnm4}, the numerator of $m_4$ is positive for $-\sqrt{3}<a<0$ and it is negative for $a< -\sqrt{3}$. From equation \eqref{eqdm4}, the denominator of $m_4$ is positive for $a<b-\sqrt{b^2+1}$ and it is negative for $a>b-\sqrt{b^2+1}$. Thus  in order to have $m_4>0$,  $-\sqrt{3}<a<b-\sqrt{b^2+1}$. \\
\end{proof}

\begin{theorem}[\textbf{Nondegeneracy of convex kite  central configurations in restricted subspace}]\label{nondegrestricted}
For any  convex kite central configurations with $(a,b)\in \mathcal{R}_X$, we have $\det(Jac(G))\not=0$, i.e., they are nondegenerate in the restricted subspace.  
\end{theorem}
\begin{proof}
The proof is based on the interval algorithm. For any  convex kite central configurations with $(a,b)\in \mathcal{R}_X$, $d_{m_3}\not=0$ and $d_{m_4}\not=0.$ $\det(Jac(G)=0$ is equivalent to $$d_{m_3} d_{m_4} \det(Jac(G))=0,$$ which is in turn equivalent to  $F_{red}(a,b)= 0$ given in \eqref{Dab_reduced}.
Since $(a,b)\in \mathcal{R}_X\subset [-\sqrt{3},\sqrt{3}-2]\times[2-\sqrt{3},\sqrt{3}] \subset[-1.74,-0.26]\times[0.26, 1.74], $ we employ rigorous interval arithmetic to prove $F_{red}(a,b)>0$ for any $(a,b)\in [-1.74,-0.26]\times[0.26, 1.74].$ \\
The rectangle $[-1.74,-0.26]\times[0.26, 1.74]$ is divided into $300\time 300$ sub-rectangles $A_k\times B_j =[a_k,a_k+\delta]\times [b_j,b_j+\delta]$. On each sub-rectangle, we evaluate the interval extension of $F_{red}(A_k,$ $B_j)$ using SageMath's arbitrary-precision interval arithmetic package. Since interval arithmetic satisfies the inclusion property, the computed interval $F_{red}(A_k,B_j)$ contains all values $F_{red}(a,b)$ for all $a\in A_k, b\in B_j$. The computed intervals contains no zeros. In fact, the smallest value of $F_{red}$ among all the sub-rectangles is $155.09$. Therefore $\det(Jac(G))>0$ for all $(a,b)\in \mathcal{R}_X$. Since the reduced Jacobian is nonsingular, every convex kite central configuration in  $\mathcal{R}_X$ is nondegenerate in the restricted configuration space.
\end{proof}


\begin{theorem}[\textbf{Nondegeneracy of convex kite central configurations in full space}]\label{nondegfull}
For every convex kite central configuration with  $(a,b)\in\mathcal{R}_X,$  one has $\det(J_2)\neq 0.$ 
Consequently, every convex kite central configuration in $\mathcal{R}_X$ is nondegenerate in the full configuration space.
\end{theorem}

\begin{proof}
From \eqref{detJ2}, we have
\[\det(J_2)=\left(a_{11}a_{33}-a_{31}a_{13}\right)\left(a_{22}a_{44}-a_{42}a_{24}\right).
\]
Therefore, it is sufficient to show that neither factor vanishes in
$\mathcal{R}_X$.

By the definition of $\mathcal{R}_X$, one can verify that the transformations \eqref{Fj2h} and \eqref{Fj2v}
preserve the zero sets in the admissible region.  Since $r_a>0, r_b>0$ and $d>0,$  we obtain $C_{factorH}>0$
and $C_{factorV}>0$ throughout $\mathcal{R}_X$. In particular, $dm_3=4r_b^3\left(r_a^3-d^3\right)\neq 0$
and $dm_4=-4r_a^3\left(r_b^3-d^3\right)\neq 0.$

Consequently,
\[a_{11}a_{33}-a_{31}a_{13}=0\quad\Longleftrightarrow\quad F_{J_{2h}}(a,b)=(a_{11}a_{33}-a_{31}a_{13})C_{factorH}=0,\]
and
\[a_{22}a_{44}-a_{42}a_{24}=0 \quad\Longleftrightarrow\quad F_{J_{2v}}(a,b)=(a_{22}a_{44}-a_{42}a_{24})C_{factorV}=0 \]
for every $(a,b)\in\mathcal{R}_X$. It remains to prove that $F_{J_{2h}}(a,b)\neq 0 $ and  $F_{J_{2v}}(a,b)\neq 0$
throughout $\mathcal{R}_X$.

For this purpose, we use rigorous interval arithmetic. The admissible region
$\mathcal{R}_X$ is contained in the rectangle
\[\mathcal{B}=[-1.74,-0.26]\times[0.26,1.74].\]
We subdivide $\mathcal{B}$ into finitely many rectangular interval boxes $X=A\times B.$ 
For each box, we first use interval arithmetic to determine whether it can
intersect the closure $\overline{\mathcal{R}_X}$. More precisely, a box is
discarded whenever one of the inequalities defining
$\overline{\mathcal{R}_X}$ is rigorously violated throughout the box.

For every remaining box $X=A\times B$, we evaluate
$F_{J_{2h}}$ and $F_{J_{2v}}$ using outward-rounded real interval arithmetic.
If  $ 0\notin F_{J_{2h}}(X) $ and  $0\notin F_{J_{2v}}(X),$  then neither function can vanish at any point of  $ X\cap\overline{\mathcal{R}_X},$  and the box is certified.

If zero belongs to either interval evaluation, the box is bisected and the
same procedure is applied recursively to the resulting subboxes. The
computation terminates after finitely many subdivisions, with every box that
can intersect $\overline{\mathcal{R}_X}$ rigorously satisfying  $0\notin F_{J_{2h}}(X)$  and  $ 0\notin F_{J_{2v}}(X).$ 
 In our program, 898 boxes are checked. The smallest certified positive lower bounds are $F_{J_{2h}}> 23.57$
and $F_{J_{2v}}>229.69.$

Therefore, $\det(J_2)\neq 0$ throughout $\mathcal{R}_X$. Thus every convex kite central configuration in $\mathcal{R}_X$ is
nondegenerate in the full configuration space.
\end{proof}









\section{Concave Kite Central Configurations}
In this section, we investigate the concave kite central configurations with the configuration given in equation \eqref{KiteCon} and $0<a<b.$ Positivity of $m_3$ and $m_4$ gives us the following admissible region for concave kite central configurations. 

\begin{definition}
 Let $\mathcal{R}_{v1}$ denote the open admissible region bounded by the vertical line 
$a = 0$ and horizontal line $b = \sqrt{3}$, together with the blue curve
$ b = a + \sqrt{a^{2}+1}$ 
as illustrated in the right panel of Figure~\ref{fig2}.
$$\mathcal{R}_{v1}=\{ (a,b)|0<a<\frac{\sqrt{3}}{3}, a + \sqrt{a^{2}+1}<b<\sqrt{3} \}.$$
Let $\mathcal{R}_{v2}$ denote the open admissible region bounded by the vertical line 
$a = \sqrt{3}$ and horizontal line $b = \sqrt{3}$, together with the blue curve
$ b = a + \sqrt{a^{2}+1} $
as illustrated in the right panel of Figure~\ref{fig2}.
$$\mathcal{R}_{v2}=\{ (a,b)|\frac{\sqrt{3}}{3}<a<\sqrt{3}, \sqrt{3}<b< a + \sqrt{a^{2}+1} \}.$$

The intersection points of these boundary curves are $A_2=[0,  1],$  $B_2=[\sqrt{3},$ $ \sqrt{3}]$,
     $C_2=[\sqrt{3}, 2+\sqrt{3}]$,  $D_2=[0,\sqrt{3}],$ and $E=(\frac{\sqrt{3}}{3},\sqrt{3}).$
\end{definition}

\begin{theorem}
Assume that the kite configuration $q$ defined by \eqref{KiteCon} is a concave central configuration for the mass vector $m=[1,1,m_3,m_4],$ with $m_3>0$ and $m_4>0$. Then $(a,b)\in \mathcal{R}_{v1}\cup \mathcal{R}_{v2}\cup\{E\}.$ 
Moreover, at the exceptional point $E=\left(\frac{\sqrt{3}}{3},\sqrt{3}\right),$ 
the configuration $q$ is a central configuration with $m_4=1$ and with $m_3$ an arbitrary positive value.
\end{theorem}

\begin{proof}
The first conclusion follows directly from the sign analysis of $n_{m_3}, d_{m_3}, n_{m_4}, d_{m_4}$  in \eqref{eqR2}. Indeed, the conditions $m_3>0$ and $m_4>0$ imply that  $(a,b)$  must belong to $\mathcal{R}_{v1}\cup \mathcal{R}_{v2}\cup{E}.$

At the point  $E=\left(\frac{\sqrt{3}}{3},\sqrt{3}\right),$  we have $d_{m_4}(E)=n_{m_4}(E)\neq 0.$
Therefore, the second equation in \eqref{eqR2} gives $m_4=1.$ 
On the other hand, $d_{m_3}(E)=n_{m_3}(E)=0.$  Hence, the first equation in \eqref{eqR2} is identically satisfied at $E$, regardless of the value of $m_3$. Consequently, every $m_3>0$, together with $m_4=1$, gives a central configuration at $E$.
Geometrically, at $E$, the three bodies with masses $m_1=m_2=m_4=1$ form an equilateral triangle, while the body of mass $m_3$ is located at its center. Thus, the central-configuration equations hold for every positive value of $m_3$.

Although the quotient formula for $m_3$ in \eqref{eqm34} is undefined at $E$, its limiting value depends on the direction of approach. More precisely, along the path  $(a,b)=\left( \frac{\sqrt{3}}{3}+\varepsilon, \sqrt{3}+k\varepsilon\right),$ as $ \varepsilon\to 0,$ we have
\begin{equation}\label{limit_m3}
    \lim_{\varepsilon\to 0} m_3(a,b) =\frac{\sqrt{3}k}{9-6k}.
\end{equation} 
As $k$ varies from $0$ to $\frac{3}{2}$, this limit varies from $0$ to  $+\infty$. This is consistent with the fact that, at $E$, $m_3$ may be chosen arbitrarily in  $(0,\infty)$ .
\end{proof}

\begin{theorem}
     Assume that the kite configuration $q$ defined by \eqref{KiteCon} is a concave kite  central configuration for mass vector $m=[1,1, m_3, m_4]$ with $m_3>0$ and $m_4>0$. Then there exists a constant $m_{4\max}>0$ such that $m_4\leq m_{4{\max}}.$
\end{theorem}
\begin{proof}
  By the expressions of $n_{m_4}$ and $d_{m_4}$ in \eqref{eqnm4} and \eqref{eqdm4}, respectively, we have $n_{m_4}\leq 0$ and $d_{m_4}<0$  for every $(a,b)\in \overline{\mathcal{R}{v_1}\cup \mathcal{R}{v_2}}.$ Moreover, $d_{m_4}$ never vanishes on this closed region. Hence  $m_4=\frac{n_{m_4}}{d_{m_4}}$ is a continuous function on the compact set $\overline{\mathcal{R}{v_1}\cup \mathcal{R}{v_2}}$. By the Extreme Value Theorem, $m_4$ attains both its  minimum and  maximum on this set. Since $n_{m_4}\le0$ and $d_{m_4}<0$, we have $m_4\ge0.$ Furthermore, $n_{m_4}=0$ on the boundary, so the minimum value is $0$. Let
$$m_{4{\max}}= \max_{(a,b)\in\overline{\mathcal{R}_{v1}\cup\mathcal{R}_{v2}}}m_4(a,b).$$  Then every concave kite central configuration satisfies $ 0<m_4\le m_{4{\max}},$  which completes the proof.
\end{proof}
 \begin{figure}[ht]
    \centering
       \includegraphics[width=0.6\textwidth,height=3in]{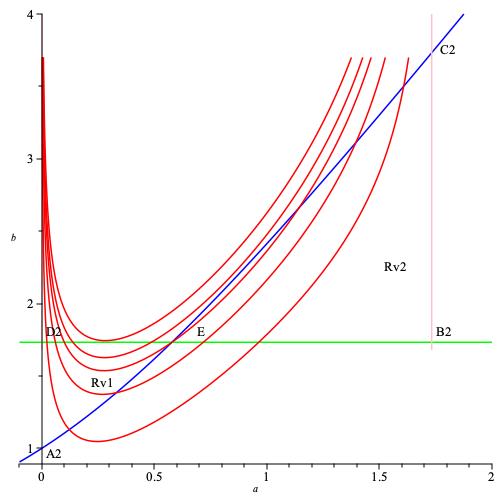}
     \caption{Boundedness of $m_4$ in the region $\overline{\mathcal{R}{v_1}\cup \mathcal{R}{v_2}}$. $0\leq m_4\leq m_{4\max}.$ Curves from bottom to top corresponding to $m_4=0.3, 0.7, 1, 1.2, 1.5$. } 
     \label{fig4}
 \end{figure}
\begin{remark}
    Numerical optimization shows that $m_4$ attains its absolute maximum on the boundary $b=\sqrt{3}$. More precisely,
$$m_{4{\max}}=  \max_{(a,b)\in \overline{\mathcal{R}_{v1}\cup\mathcal{R}_{v2}}} m_4(a,b) \approx 1.468895553993699,
$$
and the maximum is attained at   $(a,b) \approx\left(0.280570316717107,\sqrt{3}\right).$ At this point, $m_3=0.$

\end{remark}

\begin{corollary}
If the concave configuration $q$ defined in \eqref{KiteCon} satisfies $(a,b)\notin \overline{\mathcal{R}{v_1}\cup \mathcal{R}{v_2}}$, then the value of $m_3$ given by \eqref{eqm34} is negative. 
\end{corollary}

\subsection{Degeneracy in Restricted Kite Symmetrical Subspace}
\begin{figure}[ht]
    \centering
    \includegraphics[width=0.48\textwidth,height=3in]{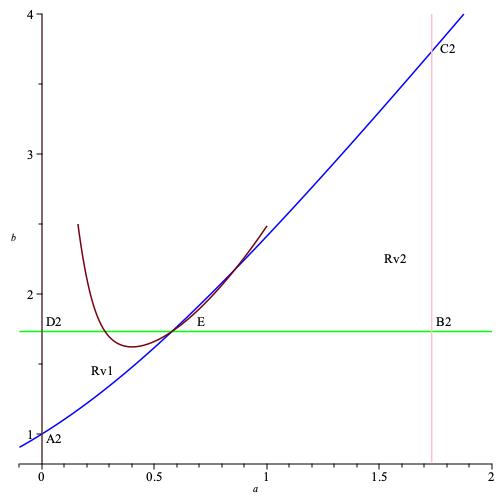}
     \includegraphics[width=0.48\textwidth,height=3in]{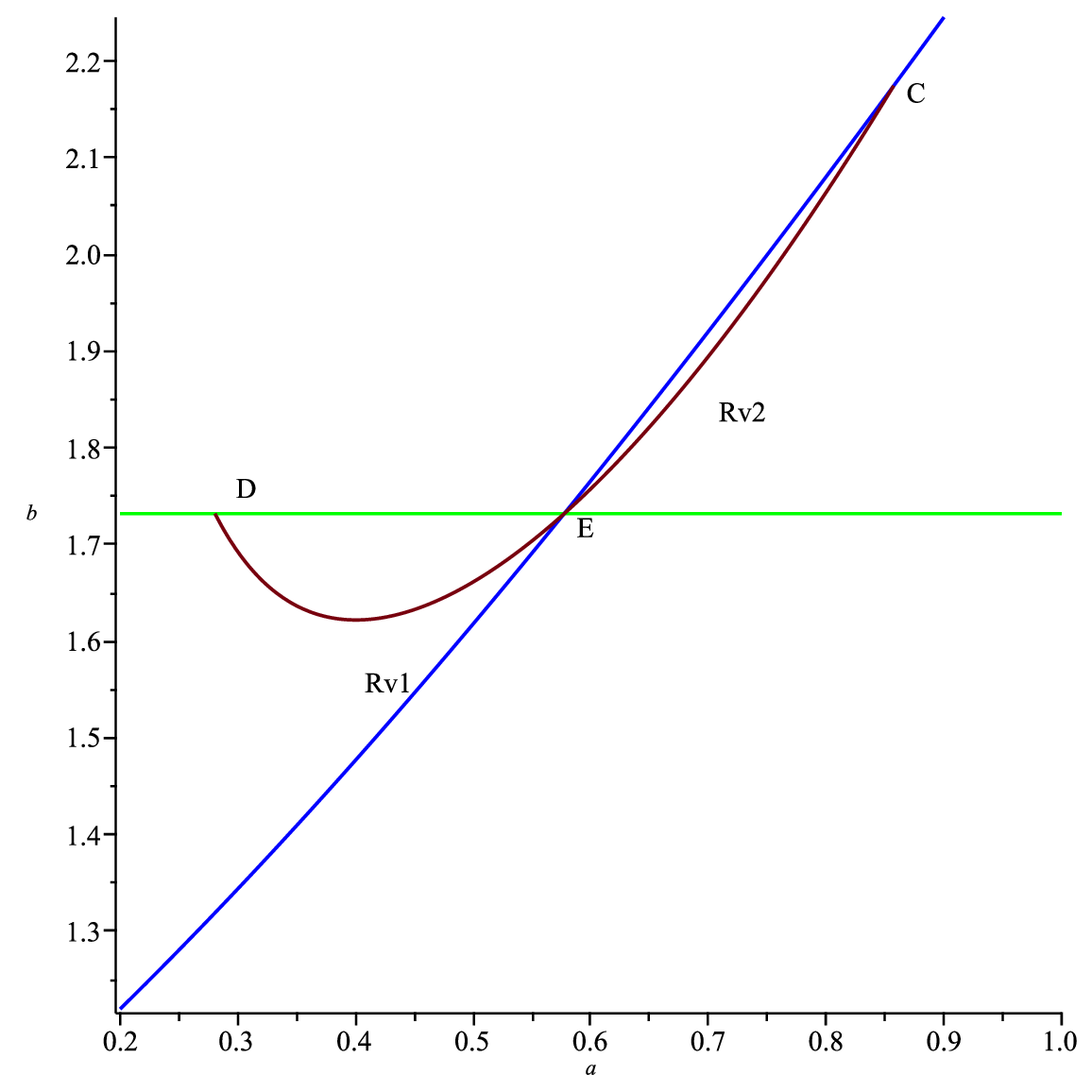}
    \caption{Degeneracy of concave kite central configuration $\det(Jac(G))=0$ in restricted subspace: red curve in region $\mathcal{R}_{v1}$ and $\mathcal{R}_{v2}$.}
    \label{fig_deg_redS}
\end{figure}


\begin{theorem}[\textbf{Degeneracy of concave kite central configurations in the restricted subspace}]\label{thm:DegReduced}
Let $F_{red}(a,b)=d_{m_3}d_{m_4}\det\bigl(\operatorname{Jac}(G(a,b))\bigr)$ be defined as in \eqref{Dab_reduced}.  Then the degeneracy set
$$\Gamma=\{(a,b):F_{red}(a,b)=0\} \subset \mathcal{R}_{v1}\cup\mathcal{R}_{v2}\cup\{E\}$$ is a unique smooth curve connecting $$D=\left(0.2805703167,\sqrt{3}\right) \quad\text{and}\quad C=\left(0.8573000157,\,2.174480078\right), $$
and passing through $E=\left(\frac{\sqrt3}{3},\sqrt3\right). $ For every point $(a,b)\in\Gamma$, the corresponding concave kite central configuration is degenerate in the restricted subspace (See Figure~\ref{fig_deg_redS}). 
\end{theorem}

\begin{proof}
The proof combines interval arithmetic, the Krawczyk operator, and the quantitative implicit function theorem.

Choose an initial value $a_0.$ Applying the one-dimensional Krawczyk operator to the equation
$F_{red}(a_0,b)=0$ produces an interval $B_0=[b_0^-,b_0^+]$  containing a unique zero $b_0\in B_0.$

Moreover, interval evaluation verifies that $0\notin \frac{\partial F_{red}}{\partial b}(a_0,B_0),$  so the partial derivative with respect to $b $ is uniformly bounded away from zero on \(B_0\).

Next, interval arithmetic provides rigorous bounds for
\[ \left|\frac{\partial F_{red}}{\partial a}\right|,\qquad \left|\frac{\partial^2 F_{red}}{\partial a^2}\right|,\qquad \left|\frac{\partial^2 F_{red}}{\partial a \partial b}\right|,\qquad \left|\frac{\partial^2 F_{red}}{\partial b^2}\right| \]
on a rectangle containing \((a_0,B_0)\).
These bounds satisfy the hypotheses of the quantitative implicit function Theorem~5 in
\cite{Sun2023}. Hence there exist positive constants \(r\) and \(\varepsilon\) such that, for every
\[
a\in[a_0-\varepsilon,a_0+\varepsilon],
\]
there exists a unique solution
\[
b=g(a)\in(b_0-r,b_0+r)
\]
satisfying $F_{red}(a,b)=0.$ 
Therefore, the local zero set of \(F_{red}\) is the graph of a unique smooth function.

To continue this curve globally, we perform a validated continuation in the parameter \(a\). Suppose the solution has been verified up to \(a=a_i\). Using the enclosure of \(g(a_i)\) obtained in the previous step as the initial interval, the Krawczyk operator computes a new interval \(B_i\) containing the unique solution $b_i$ of \[ F_{red}(a_i,b_i)=0, \] and verifies that
\[ 0\notin \frac{\partial F_{red}}{\partial b}(a_i,B_i). \]
The interval bounds for the first and second derivatives are then recomputed on a neighborhood of \((a_i,B_i)\), and the hypotheses of Theorem~5 in \cite{Sun2023} are verified again. Consequently, there exists a neighborhood
\[
[a_i-\varepsilon_i,a_i+\varepsilon_i]
\]
on which the implicit solution is uniquely defined. Since consecutive neighborhoods overlap, uniqueness implies that the locally defined solution segments coincide on their intersections, producing a single smooth curve.

The continuation proceeds until it reaches the two boundary points
\[
D=\left(0.2805703167,\sqrt3\right)
\quad\text{and}\quad
C=\left(0.8573000157,\,2.174480078\right),
\]
passing through
\[
E=\left(\frac{\sqrt3}{3},\sqrt3\right).
\]
The validated continuation consists of finitely many overlapping neighborhoods, since the interval for $a$ is compact. Finally we get a unique smooth curve
\[
\Gamma=\{(a,b):F_{red}(a,b)=0\},
\]
which is precisely the connected component of the degeneracy set contained in
\[
\mathcal{R}_{v1}\cup\mathcal{R}_{v2}\cup\{E\}.
\]
The smoothness of \(\Gamma\) follows from the classical implicit function theorem because
\[
\frac{\partial F_{red}}{\partial b}(a,g(a))\neq0
\]
throughout the whole process.
\end{proof}
\begin{remark}
\begin{enumerate}
    \item Although the boundary degeneracy points $D$ and $C$ were computed by interval arithmetic and Krawczyk operator, one can solve the equation on the boundaries. For $b=\sqrt{3}$: $F_{red}(a,\sqrt{3})=0$ for $0< a< \sqrt{3}$ which has two solutions $a=0.2805703167$ and $a=\frac{\sqrt{3}}{3}.$ These give the degenerate point $D$ and $E$.
    On the boundary $b=a+\sqrt{a^2+1}$: $F_{red}(a, a+\sqrt{a^2+1})=0$ for $0< a< \sqrt{3}$ has two solutions $a=0.8573000157$ and $a=\frac{\sqrt{3}}{3}.$ These give the degenerate point $C$ and $E$.
    \item As an illustration, we present one representative example. Let $a_0=0.5.$ Applying the Krawczyk operator yields the verified solution $b_0=1.6614016692$. Using the quantitative implicit function theorem, we obtain the validated bounds 
    $$r_0=0.0059083661, \qquad \epsilon_0=0.0016780322.$$ Consequently, there exists a unique solution \(b=g(a)\) for every $a\in[a_0-\epsilon_0,a_0+\epsilon_0],$  with  $b\in[b_0-r_0,b_0+r_0].$  That is, the solution curve is uniquely contained in the rectangle
    $$[a_0-\epsilon_0,a_0+\epsilon_0] \times [b_0-r_0,b_0+r_0].$$
To verify that no additional solutions exist in the same $a$-interval, we apply interval arithmetic to the reduced equation $F_{\mathrm{red}}(a,b)$ . Fixing $a\in[a_0-\epsilon_0,a_0+\epsilon_0]$, we rigorously prove that $F_{\mathrm{red}}(a,b)\neq0$ for all $b\notin[b_0-r_0,b_0+r_0].$  Therefore, the above rectangle contains the unique zero of $F_{\mathrm{red}}$ over the prescribed interval.

\item At point $E=(\frac{\sqrt{3}}{3}, \sqrt{3})$, $r_0=0.0061607213, \epsilon_0=0.0014306898$. We further computed the derivative $\frac{db}{da}=-\frac{\frac{\partial F_{red}(a,b)}{\partial a}}{\frac{\partial F_{red}(a,b}{\partial b}}$, which is about $1.09117428169.$ 
At $E$, this gives $g'\left(\frac{\sqrt{3}}{3}\right) \approx 1.09117428169.$  Since $0<1.09117428169<\frac{3}{2},$ the tangent to the implicit curve $F_{red}(a,b)=0$ at $E$ lies strictly between the tangents to these two boundary curves. Consequently, the implicit curve $b=g(a)$  passes through $E$, entering $\mathcal{R}_{v1}$  on one side of $E$ and $\mathcal{R}_{v2}$ on the other. Moreover, At point $E$, $m_4=1$ and $m_3=0.7704869546$ by using formula \eqref{limit_m3} with $k=1.09117428169.$ This degeneracy at $E$ is the known degeneracy of four-body problem having equilateral triangle with fourth one at the center. 

\item Similarly, at point $D=(0.2805703167,\sqrt{3}),$  the rate $\frac{db}{da}\approx -2.373726621$ which indicates that the degenerate curve cross the horizontal line $b=\sqrt{3}$. This shows that the degenerate curve enters the region $\mathcal{R}_{v1}$ \\
At point $C=(0.8573000157,2.174480078)$, the rate $\frac{db}{da}\approx 1.996600507$ which is bigger than the slope ($1.650860152)$ of the tangent line to the curve $b=a+\sqrt{a^2+1}$. This shows that the degenerate curve leaves the region $\mathcal{R}_{v2}.$

\end{enumerate}
   
\end{remark}



\subsection{Degeneracy in Full Configuration Space}
Having reduced the degeneracy problem in the full configuration space to the equations 
$$F_{J_{2v}}(a,b)=0 \quad\text{and}\quad F_{J_{2h}}(a,b)=0,$$
we now apply interval arithmetic  and the quantitative implicit function theorem  to rigorously analyze their zero sets which yields computer-assisted proofs of the existence of the corresponding degeneracy curves. 

\begin{figure}[ht]
    \centering
    \includegraphics[width=0.6\textwidth,height=3in]{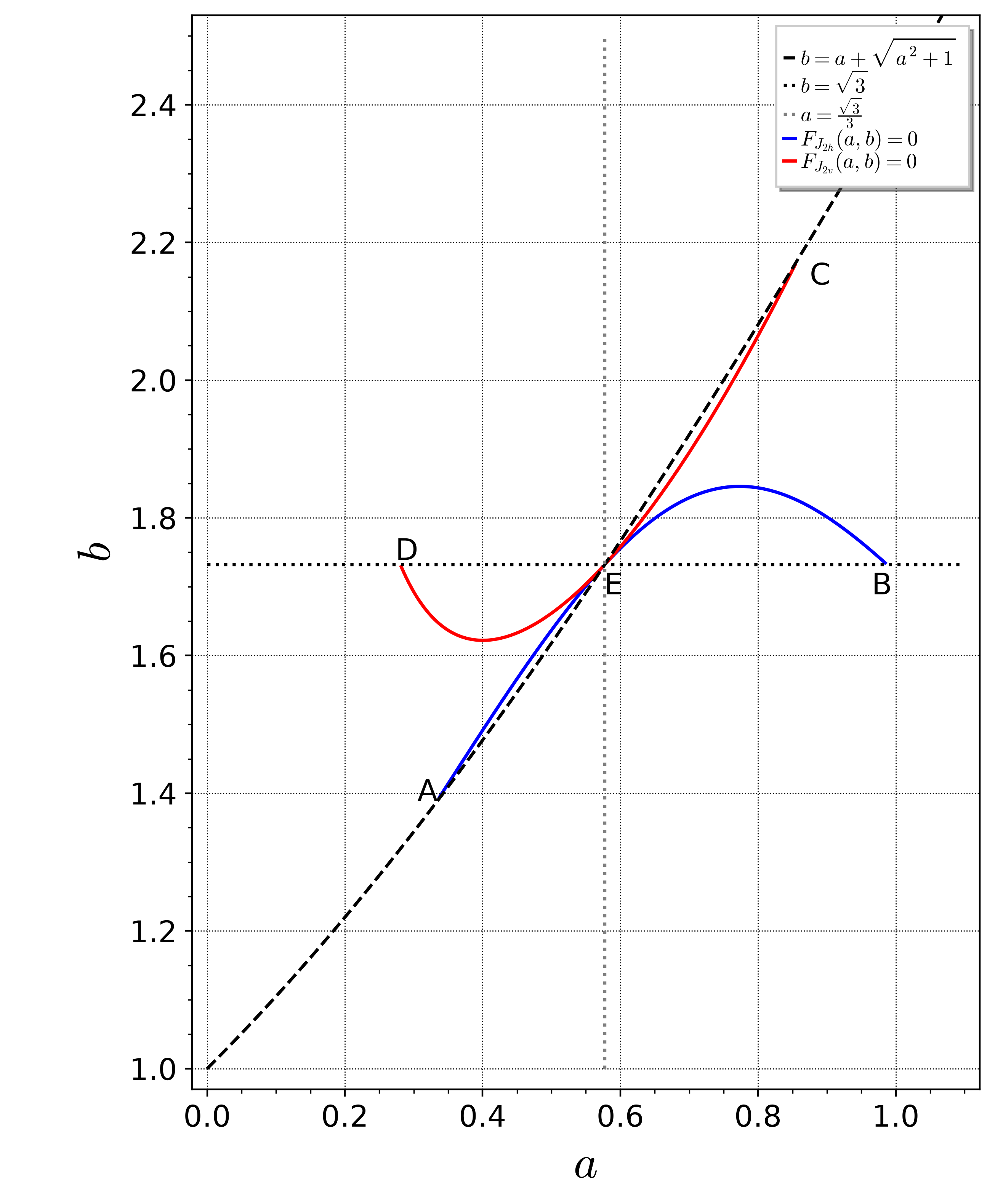}
    \caption{Degeneracy of concave central configuration $\det(J_2)=0$ in region $\mathcal{R}_{v1}$ and $\mathcal{R}_{v2}$. Red curve $DEC$ is for $F_{J_{2v}}=0$ and blue curve $AEB$ is for $F_{J_{2h}}=0.$ }
    \label{fig_Fj2h2v}
\end{figure}

\begin{theorem}[\textbf{Horizontal degeneracy curve in full space}]\label{thm:Fh}
Let $F_{J_{2h}}(a,b)=(\hat{a}_{11}\hat{a}_{33}-\hat{a}_{31}\hat{a}_{13})$ be given in equation \eqref{Fj2h}.
Then the zero set of $F_{J_{2h}}$ in the concave region $\mathcal R_{v1}\cup\mathcal R_{v2}\cup\{E\}$ 
consists of a unique connected smooth curve
 $$ \Gamma_h =\{(a,b):F_{J_{2h}}(a,b)=0\} = \{(a,b_h(a)):a_h^-< a < a_h^+\},$$
where $a_h^- \approx  0.33170680252449$ and $a_h^+ \approx 0.987397455085085$. Moreover, 

\begin{enumerate}
\item
For any $a\in(a_h^-,a_h^+)$, $\Gamma_h$ intersects the vertical line $a=\mathrm{constant}$ at exactly one point.

\item
The curve joins the two boundary components of the concave region,
with one endpoint $A= (a_h^-, a_h^-+\sqrt{{a_h^-}^2+1})$ on $b=a+\sqrt{a^2+1},$ 
and the other $B=(a_h^+,\sqrt{3})$, on $b=\sqrt3 $ (see the blue curve in Figure \ref{fig_Fj2h2v}).

\item
The curve passes through the point $E=\left(\frac{\sqrt3}{3},\sqrt3\right)$.
 \item The graph $b=b_h(a)$ is strictly concave downward: $b_h''(a)<0$ for $a\in(a_h^-,a_h^+).$ 
\item
For every $(a,b)\in\Gamma_h\setminus\{E\}$, the horizontal block of the full Jacobian has rank one, while the vertical block is nonsingular.
Consequently, the kernel of the full Jacobian is one-dimensional and lies in the horizontal subspace.
\end{enumerate}
\end{theorem}

\begin{proof}
Recall that \[ F_{J_{2h}}(a,b) = \hat a_{11}\hat a_{33}-\hat a_{31}\hat a_{13}. \]
Since all denominator factors removed in the definition of
$F_{J_{2h}}$ are nonzero in the admissible region, the horizontal
block of the full Jacobian is singular if and only if $F_{J_{2h}}(a,b)=0.$ We apply rigorous interval arithmetic and the Krawczyk operator to determine the zero set of $F_{J_{2h}}$ in $\mathcal R_{v1}\cup\mathcal R_{v2}\cup\{E\}.$ 

All interval computations are carried out with outward rounding in
SageMath~10.7. We first determine the two boundary intersections.  On the curved
boundary $b=a+\sqrt{a^2+1},$ 
we consider the one-variable function
\[
H_-(a)
=
F_{J_{2h}}
\left(a,a+\sqrt{a^2+1}\right).
\]
The interval Krawczyk operator gives exactly two zeros
\[
a_h^- = 0.33170680252448958\ldots. \hbox{ and } a_E=0.5773502693\ldots =\frac{\sqrt{3}}{3}, 
\]
Thus the corresponding left endpoint is $A= \left( a_h^-, a_h^-+\sqrt{(a_h^-)^2+1}\right).$ 
To determine the direction in which the zero curve crosses the
boundary, introduce $s=b-a-\sqrt{a^2+1}$ 
and set $G_-(a,s) = F_{J_{2h}} \left(a,a+\sqrt{a^2+1}+s\right).$
A rigorous interval evaluation in a neighborhood of $(a_h^-,0)$ gives
$\frac{\partial G_-}{\partial s}\neq0,-\frac{G_{-,a}}{G_{-,s}}>0.$ 
Hence, by the implicit function theorem, the zero curve crosses the
boundary transversally at $A$ and enters $\mathcal R_{v1}$ as $a$
increases.

Similarly, on the boundary $b=\sqrt3$, the function $H_+(a)=F_{J_{2h}}(a,\sqrt3)$  has, by interval Newton verification, two zeros
\[
a_h^+ = 0.9873974550850850250\ldots \hbox{ and } a_E=0.5773502693\ldots =\frac{\sqrt{3}}{3}, 
\]
Thus the right endpoint is $B=(a_h^+,\sqrt3).$  Introducing $t=b-\sqrt3$ and $G_+(a,t)=F_{J_{2h}}(a,\sqrt3+t),$ 
the interval calculation gives
\[
\frac{\partial G_+}{\partial t}\neq0,
\qquad
-\frac{G_{+,a}}{G_{+,t}}<0.
\]
Therefore the zero curve approaches $B$ transversally from the
admissible side of $\mathcal R_{v2}$.

We next rigorously continue the zero curve between these two endpoint
neighborhoods.  The interval  $(a_h^-,a_h^+)$ is covered by $1000$ vertical parameter strips $A_i$.  For each
$A_i$, an interval $B_i$ is constructed around a numerical predictor and 
the Krawczyk  interval $K_i$  is verified to satisfy $K_i\subset\operatorname{int}(B_i).$  
Consequently, for every $a\in A_i$ there exists a unique
$b\in B_i$ satisfying $F_{J_{2h}}(a,b)=0.$ 
Moreover, rigorous interval evaluation proves $\partial_bF_{J_{2h}}(a,b)\neq0$ on every admissible vertical fiber covered by these strips; in fact, the computation gives the uniform certified bound
$\left|\partial_bF_{J_{2h}}(a,b)\right| >5.7417\times10^4.$ 
Hence no additional zero can occur on any admissible vertical fiber.
The implicit function theorem therefore shows that these locally
validated pieces form a unique connected smooth graph
\[
\Gamma_h = \left\{ (a,b_h(a)): a_h^-<a<a_h^+ \right\}.
\]

At the special point $E=\left(\frac{\sqrt3}{3},\sqrt3\right),$  direct symbolic substitution gives exactly $F_{J_{2h}}(E)=0.$ Thus $E\in\Gamma_h$.  In addition, $\partial_bF_{J_{2h}}(E)\neq0,$ 
and interval evaluation of the implicit derivative gives $b_h'(E)=-\frac{\partial_aF_{J_{2h}}(E)}{\partial_bF_{J_{2h}}(E)} \approx1.09117.$

As for the concavity of the curve,  the second derivative of the implicit graph is given by
\[ b''(a) =-\frac{ \partial_{aa}F_{J_{2h}} +2\partial_{ab}F_{J_{2h}}b'(a)+\partial_{bb}F_{J_{2h}}(b'(a))^2 }{\partial_{b}F_{J_{2h}}}, \qquad b'(a)=-\frac{\partial_{a}F_{J_{2h}}}{\partial_{b}F_{J_{2h}}}. \]
Using the validated root tubes obtained above, rigorous interval
evaluation of this expression proves $b_h''(a)<0$
throughout $(a_h^-,a_h^+)$ for the horizontal branch.  Hence
$\Gamma_h$ is strictly concave downward.

It remains to determine the kernel of the full Jacobian along $\Gamma_h$.  Rigorous interval evaluation verifies that
\[F_{J_{2v}}(a,b)\neq0 \qquad \text{for }(a,b)\in\Gamma_h\setminus\{E\}, \] except that a separate local argument is required near $E$, where
the horizontal and vertical degeneracy curves are tangent.  At $E$,
symbolic computation gives $F_{J_{2h}}(E)=F_{J_{2v}}(E)=0$
and $b_h'(E)=b_v'(E).$ 
Second-order interval implicit differentiation in a neighborhood of
$E$ proves that $b_h''(a)-b_v''(a)$  has a strict nonzero sign there.  Hence the two degeneracy curves have
no other intersection in this neighborhood, and therefore
\[
\Gamma_h\cap\{F_{J_{2v}}=0\}=\{E\}.
\]

Finally, interval evaluation also verifies that at least one entry of
the horizontal $2\times2$ block is nonzero at every point of $\Gamma_h$.  Since $F_{J_{2h}} = \det(J_{2h})=0\,\, \text{on }\Gamma_h,$  the horizontal block has rank exactly one.  For $(a,b)\in\Gamma_h\setminus\{E\}$ the vertical block is nonsingular, since $F_{J_{2v}}\neq0$.  Consequently, the kernel of the full Jacobian is one-dimensional and lies in the horizontal subspace.

This proves all the assertions.
\end{proof}

\begin{theorem}[\textbf{Vertical degeneracy curve in full space}]\label{thm:Fv}
Let $F_{J_{2v}}(a,b)=(\hat{a}_{22}\hat{a}_{44}-\hat{a}_{24}\hat{a}_{42})$ be given in equation \eqref{Fj2v}.
Then the zero set of $F_{J_{2v}}$ in the concave region $\mathcal R_{v1}\cup\mathcal R_{v2}\cup\{E\}$ 
consists of a unique connected smooth curve
 $$ \Gamma_v =\{(a,b):F_{J_{2v}}(a,b)=0\} = \{(a,b_v(a)):a_v^-< a< a_v^+\},$$
where $a_v^- \approx 0.2805703167 $ and $a_v^+ \approx 0.8573000157 $. Moreover,
\begin{enumerate}
\item
For any $a\in(a_v^-,a_v^+)$, $\Gamma_v$ intersects the vertical line $a=\mathrm{constant}$ at exactly one point.

\item
The curve joins the two boundary components of the concave region,
with one endpoint $C= \left(a_v^+, a_v^++\sqrt{{a_v^+}^2+1}\right)$ on $b=a+\sqrt{a^2+1},$ 
and the other endpoint $D=\left(a_v^-,\sqrt{3}\right)$, on $b=\sqrt3 $ (see the red curve in Figure \ref{fig_Fj2h2v}).

\item
The curve passes through the point $E=\left(\frac{\sqrt3}{3},\sqrt3\right)$.
\item
For every $\left(a,b\right)\in\Gamma_v\setminus\{E\}$, the vertical block of the full Jacobian has rank one, while the horizontal block is nonsingular.
Consequently, the kernel of the full Jacobian is one-dimensional and lies in the vertical subspace.
\end{enumerate}
\end{theorem}

\begin{proof}
Recall that
\[ F_{J_{2v}}(a,b) = \hat a_{22}\hat a_{44}-\hat a_{24}\hat a_{42}.\]
Since the denominator factors removed in the definition of $F_{J_{2v}}$ are nonzero in the admissible region, the vertical block
of the full Jacobian is singular if and only if $F_{J_{2v}}(a,b)=0. $

We use rigorous interval arithmetic and the interval Krawczyk operator to
determine the zero set of $F_{J_{2v}}$ in $\mathcal R_{v1}\cup\mathcal R_{v2}\cup\{E\}.$
All interval computations are performed with outward rounding in
SageMath~10.7.

We first determine the two boundary intersections. On the boundary
$b=\sqrt3$, define $ V_-(a)=F_{J_{2v}}(a,\sqrt3). $ 
The interval Krawczyk operator gives exactly two zeros
\[a_v^-= 0.2805703167\ldots \hbox{ and } a_E=0.5773502693\ldots =\frac{\sqrt{3}}{3}. 
\]
Hence the corresponding endpoint is $D=(a_v^-,\sqrt3).$ 
Introducing $t=b-\sqrt3$
and  $G_-(a,t)=F_{J_{2v}}(a,\sqrt3+t), $ 
rigorous interval evaluation in a neighborhood of $(a_v^-,0)$ gives
\[ \frac{\partial G_-}{\partial t}\neq0, \qquad -\frac{G_{-,a}}{G_{-,t}}<0.
\]
Therefore the zero curve crosses $b=\sqrt3$ transversally at $D$ and
enters $\mathcal R_{v1}$ as $a$ increases.

On the curved boundary  $b=a+\sqrt{a^2+1}, $ 
define  $V_+(a) = F_{J_{2v}} \left(a,a+\sqrt{a^2+1}\right).$ 
The interval Krawczyk operator gives exactly two zeros
\[ a_v^+ = 0.8573000157\ldots  \hbox{ and } a_E=0.5773502693\ldots =\frac{\sqrt{3}}{3}.\]
Thus the second endpoint is $C= \left( a_v^+, a_v^++\sqrt{(a_v^+)^2+1} \right).$
Let $ s=b-a-\sqrt{a^2+1}$ 
and $G_+(a,s) = F_{J_{2v}} \left(a,a+\sqrt{a^2+1}+s\right).$ 
The interval calculation gives
\[ \frac{\partial G_+}{\partial s}\neq0, \qquad -\frac{G_{+,a}}{G_{+,s}}>0.
\]
Hence the zero curve reaches the blue boundary transversally from the
admissible side of $\mathcal R_{v2}$.

We next rigorously continue the zero set between the two endpoint
neighborhoods. The interval  $(a_v^-,a_v^+)$ 
is covered by $1000$ vertical parameter strips $A_i$. For each strip,
an interval $B_i$ is chosen around a numerical predictor and  the Krawczyk  interval $K_i$  is verified to satisfy 
$K_i\subset\operatorname{int}(B_i).$

Thus, for every $a\in A_i$, there exists a unique
$b\in B_i$ satisfying $F_{J_{2v}}(a,b)=0.$ 
Any possible additional zero on the same admissible vertical fiber is excluded either by a rigorous monotonicity test for
$\partial_bF_{J_{2v}}$ or by direct interval zero exclusion outside the certified Newton tube. Hence every admissible vertical fiber with
$a\in(a_v^-,a_v^+)$ contains exactly one zero of $F_{J_{2v}}$.

Moreover, $\partial_bF_{J_{2v}}\neq0$ on every certified root tube. The implicit function theorem therefore
shows that the validated local pieces form a unique connected smooth
graph $\Gamma_v = \left\{ \left(a,b_v(a)\right): a_v^-<a<a_v^+ \right\}.$

At $E=\left(\frac{\sqrt3}{3},\sqrt3\right),$  direct symbolic substitution gives exactly $F_{J_{2v}}(E)=0.$ 
Thus $E\in\Gamma_v$. In addition, $\partial_bF_{J_{2v}}(E)\neq0,$  so the curve is smooth through $E$.

As for the concavity, similar computations show that $b_v''(a)>0$ throughout $(a_v^-,a_v^+)$ for the vertical branch. Therefore $\Gamma_v$ is strictly convex upward.

It remains to determine the kernel of the full Jacobian along $\Gamma_v$. Rigorous interval evaluation verifies that
$F_{J_{2h}}(a,b)\neq0 $ for $ (a,b)\in\Gamma_v\setminus\{E\},$   except that a separate local argument is required near $E$, where the
horizontal and vertical degeneracy curves are tangent. The argument is provided in the proof of Theorem~\ref{thm:Fh}.

Finally, interval evaluation verifies that at least one entry of the
vertical $2\times2$ block is nonzero at every point of $\Gamma_v$. Since
$F_{J_{2v}}b= \det(J_{2v})=0$ on  $\Gamma_v,$  the vertical block has rank exactly one. For
$(a,b)\in\Gamma_v\setminus\{E\}$, the horizontal block is nonsingular
because $F_{J_{2h}}\neq0$. Consequently, the kernel of the full
Jacobian is one-dimensional and lies in the vertical subspace.

This proves all the assertions.
\end{proof}

\begin{remark}[\bf The higher degeneracy at the equilateral configuration]
\label{rem:three-equal-m124}
The highly symmetrical configuration $E=\left(\frac{\sqrt{3}}{3},\sqrt{3}\right)$, which corresponds to an equilateral triangle 
with the vertices
\[
q_1=(-1,0),\qquad
q_2=(1,0),\qquad
q_4=(0,\sqrt{3}),
\]
 while $q_3=\left(0,\frac{\sqrt{3}}{3}\right)$ 
is located at its center, plays a prominent role in the whole picture. By \eqref{eqm34}, $m_4=1$ but $m_3$ can be any positive value. There are different ways to study this configuration. 

First, the point $E$ is particularly interesting from the viewpoint of full-space
degeneracy. Substituting $m_4=1$, $(a,b)=\left(\frac{\sqrt{3}}{3},\sqrt{3}\right)$ into \eqref{detJ2} for $\det{J_2}$ with $m_3$ as a free parameter, one obtains 
$$ (a_{11}a_{33}-a_{13}a_{31})=(a_{22}a_{44}-a_{24}a_{42})= \frac{\left(135 m_3 +81\right) \sqrt{3}}{256}-\frac{243 m_3}{128}+\frac{27}{128},$$ which has zero $m_3^*= \frac{27}{83}+\frac{64 \sqrt{3}}{249}\approx 0.7704869546.$\\

On the other hand, both the horizontal and vertical degeneracy conditions are
satisfied:
\[
F_{J_{2h}}(E)=F_{J_{2v}}(E)=0.
\]
Moreover, {implicit differentiation} of the two degeneracy equations gives the
same tangent slope at \(E\):
\[
\left.
\frac{db}{da}
\right|_{F_{J_{2h}}=0,E}
=
\left.
\frac{db}{da}
\right|_{F_{J_{2v}}=0,E}
\approx
1.09117428169.
\]
Hence the horizontal and vertical degeneracy loci meet tangentially at \(E\),
rather than intersects transversally. {Moreover, the three masses at the vertices are
equal, $m_1=m_2=m_4=1,$ 
whereas the mass at the center is $m_3^*\approx 0.7704869546,$ 
as computed from \eqref{limit_m3} with slope $k=1.09117428169$. Thus, \(E\) corresponds to a highly
symmetric central configuration consisting of three equal masses at the
vertices of an equilateral triangle and a fourth mass at its center.}

This tangency reveals the enhanced symmetry of the configuration
at \(E\). At a generic point of the horizontal degeneracy locus
\(\Gamma_h\), only the horizontal block \(J_{2h}\) is singular, and the
corresponding kernel direction breaks the reflection symmetry of the kite.
On the other hand, at a generic point of the vertical degeneracy locus
\(\Gamma_v\), only the vertical block \(J_{2v}\) is singular, and the
corresponding kernel direction preserves the kite symmetry. At \(E\), however,
both blocks are singular simultaneously. 
Each of these \(2\times2\) blocks has rank one at \(E\)  as verified directly, which implies that
\[
\dim\ker J_2(E) = \dim\ker J_{2h}(E)+\dim\ker J_{2v}(E) = 2.
\]
Thus, besides the trivial directions associated with translation, rotation,
and scaling symmetries which have been singled out, the configuration at \(E\) possesses two
independent nontrivial degeneracy directions: one symmetry-preserving and one
symmetry-breaking. This agrees with the multiplicity-two degeneracy of the
nontrivial zero eigenvalue reported in \cite{Shi2010}.

The configuration \(E\) therefore provides a natural meeting point between
the symmetric and symmetry-breaking degeneracy structures. The coincidence
of the two degeneracy conditions and their tangent directions at precisely
this higher-symmetry configuration is consistent with the general connection
between symmetry and bifurcation in the classical studies
\cite{Meyer1987,Meyer1988,Santoprete2015}. {Moreover, we build the analysis of \cite{Meyer1988} into the mass parameter family, and we confirm the uniqueness of degeneracy found in \cite{Santoprete2015}.} It also complements the analysis of
kite central configurations in \cite{Roberts2025} {where degeneracy was studied in our restricted kite symmetry setting numerically}, and {the three-equal-mass
concave kite family studied in \cite{Liu2026_2}}. In this sense, \(E\) is not
merely an isolated degenerate configuration, but a distinguished point with
higher-symmetry where the symmetric and symmetry-breaking degeneracy
loci become tangential and the nontrivial nullity increases from one to two.

\end{remark}

In the following we put two more cases encountered often in the literature into our framework.

\begin{remark}[\textbf{Case 1: three equal masses with} $\mathbf{m_1=m_2=m_3=1}$]
\label{rem:three-equal-m123}
Solving equations $n_{m_3}(a,b)=d_{m_3}(a,b)$ and $F_{J_{2h}}(a,b)=0$, or equations $n_{m_3}(a,b)=d_{m_3}(a,b)$ and $F_{J_{2v}}(a,b)=0$, we obtain two solutions. The two degenerate central configurations in this case correspond to two different
bifurcation mechanisms.  Numerically, they occur at
\[
(a,b,m_3,m_4) \approx (0.5412989378,\,1.690172439,\,1,\,0.9918422748)
\]
and
\[
(a,b,m_3,m_4) \approx (0.6134917485,\,1.774177390,\,1,\,1.002660546).
\]

For the first degenerate configuration, the degeneracy occurs in the horizontal
block of $J_2$. The corresponding kernel direction breaks the reflection
symmetry of the kite, and the symmetric branch bifurcates into two
non-symmetric branches related to each other by reflection. Thus, this point
represents a symmetry-breaking bifurcation in the full planar configuration
space. This phenomenon is consistent with the general equivariant
bifurcation picture for central configurations developed in
\cite{Santoprete2015}. In particular, degeneracy in the
full configuration space may produce branches that are not visible when the
problem is restricted a priori to the kite-symmetric subspace. Similar
symmetry-breaking phenomena for concave kite configurations with three equal
masses are described in \cite{Liu2026_2}.

The second degenerate configuration has a different character. It occurs on
the symmetric branch and corresponds to a fold-type bifurcation. As $m_4$
passes through the critical value
\[
m_4^*\approx1.002660546,
\]
the two symmetric solutions coalesce and disappear; consequently, there are
no symmetric concave kite central configurations of this family beyond this
critical mass value. {This is a quite different phenomenon from the convex case where one always has symmetric central configuration for $m_1=m_2=m_3$, and the intuition that equal masses somehow imply symmetric central configurations fails for the concave case.} The bifurcation is the standard saddle-node or fold mechanism
appearing in one-parameter families of central configurations, as discussed
in \cite{Liu2026_2,Santoprete2015}.

\end{remark}

\begin{remark}[\textbf{Case 2: two pairs of equal masses with} $\mathbf{m_1=m_2=1}$ \textbf{and} $\mathbf{m_3 = m_4 =}$ $\boldsymbol{\mu}$]
\label{rem:two-pairs-equal}

Solving equations $n_{m_3}(a,b)d_{m_4}(a,b)=d_{m_3}(a,b)n_{m_4}(a,b)$ and $F_{J_{2h}}(a,b)=0$, or equations $n_{m_3}(a,b)d_{m_4}(a,b)=d_{m_3}(a,b)n_{m_4}(a,b)$  and $F_{J_{2v}}(a,b)=0$, we obtain two solutions.
The two degenerate configurations obtained under the constraint $m_3=m_4=\mu$
are numerically
\[
(a,b,\mu)\approx(0.5423375242,\,1.691445724,\,0.99229944)
\]
and
\[
(a,b,\mu)\approx(0.6138576373,\,1.774630544,\,1.00271336).
\]

These two points again correspond to two distinct types of degeneracy. The
first lies on the horizontal degeneracy locus and is associated with a
symmetry-breaking bifurcation in the full planar configuration space. The
second lies on the vertical degeneracy locus and corresponds to the
fold-type bifurcation of the symmetric family. In particular, the latter
occurs at the critical mass ratio
\[
\mu_0\approx1.002713329,
\]
at which the two symmetric concave kite configurations coalesce. For
$0<\mu<\mu_0$ there are two symmetric configurations, at $\mu=\mu_0$ there
is one degenerate configuration, and for $\mu>\mu_0$ there are no symmetric
configurations of this type.

This agrees with the complete classification of the concave kite family with
two pairs of equal masses obtained in \cite{Liu2026}, where the
critical configuration at $\mu=\mu_0$ was proved to be a fold-type
bifurcation in the symmetry-restricted subspace. The present full-space analysis
adds another aspect to that picture: before the fold is reached, the
symmetric family encounters a horizontal degeneracy near
\[
\mu\approx0.99229944,
\]
from which non-symmetric configurations may bifurcate. Thus, the reduced
analysis detects the fold of the symmetric family, while the full-space
Jacobian also detects the earlier loss of nondegeneracy in a
symmetry-breaking direction.

This distinction illustrates why it is important to study degeneracy in the
full configuration space rather than only in the kite-symmetric subspace.
It is also consistent with the general principle in equivariant bifurcation
theory that a symmetric branch may lose nondegeneracy through a mode lying
outside its fixed-point subspace; see
\cite{Meyer1987,Meyer1988,Santoprete2015}. The results therefore connect the
complete symmetric classification of \cite{Liu2026, Roberts2025} with the broader
bifurcation picture for kite central configurations in full space.
\end{remark}

\begin{figure}[ht]
    \centering
    \includegraphics[width=0.48\textwidth,height=3in]{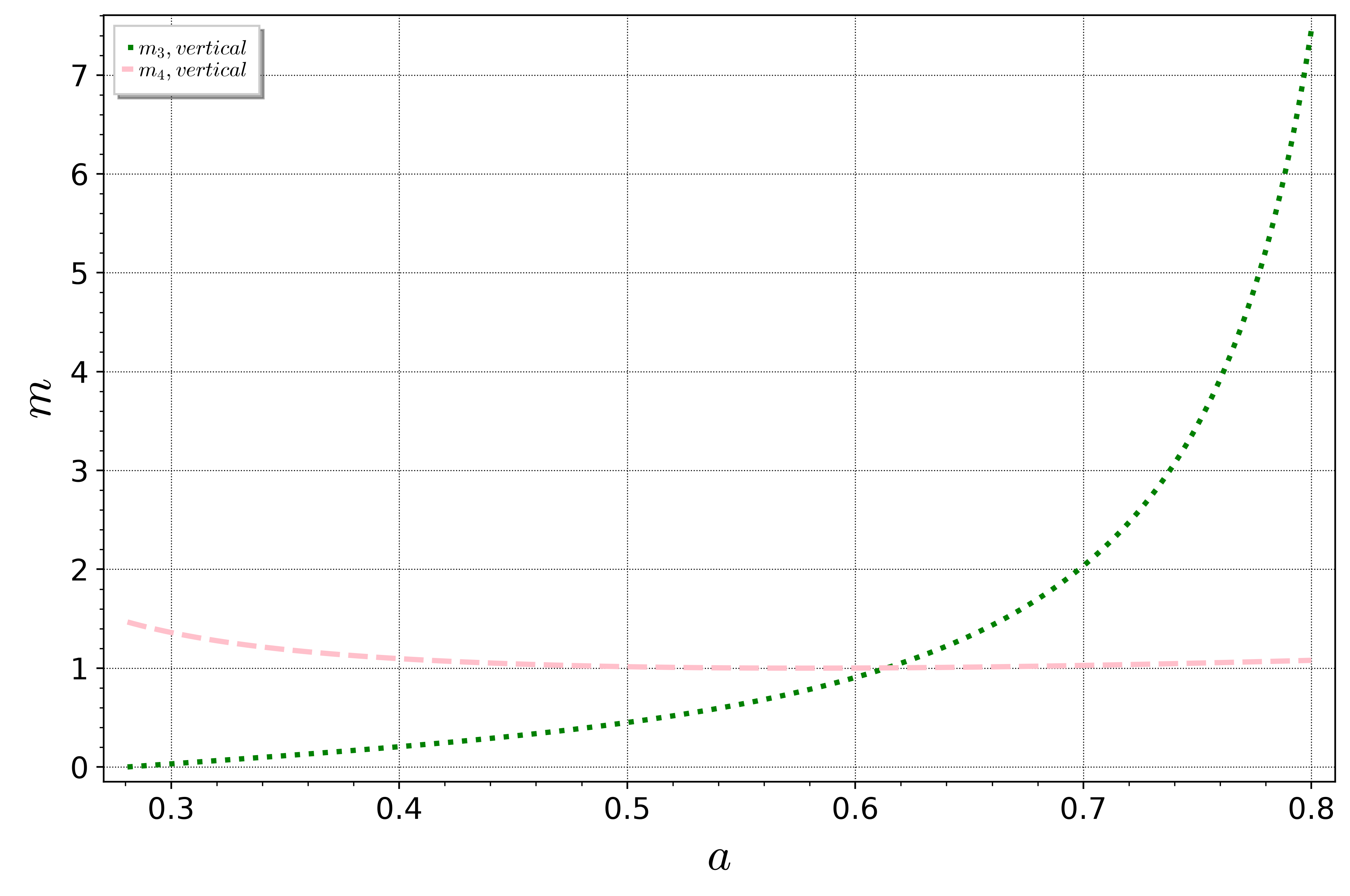}
     \includegraphics[width=0.48\textwidth,height=3in]{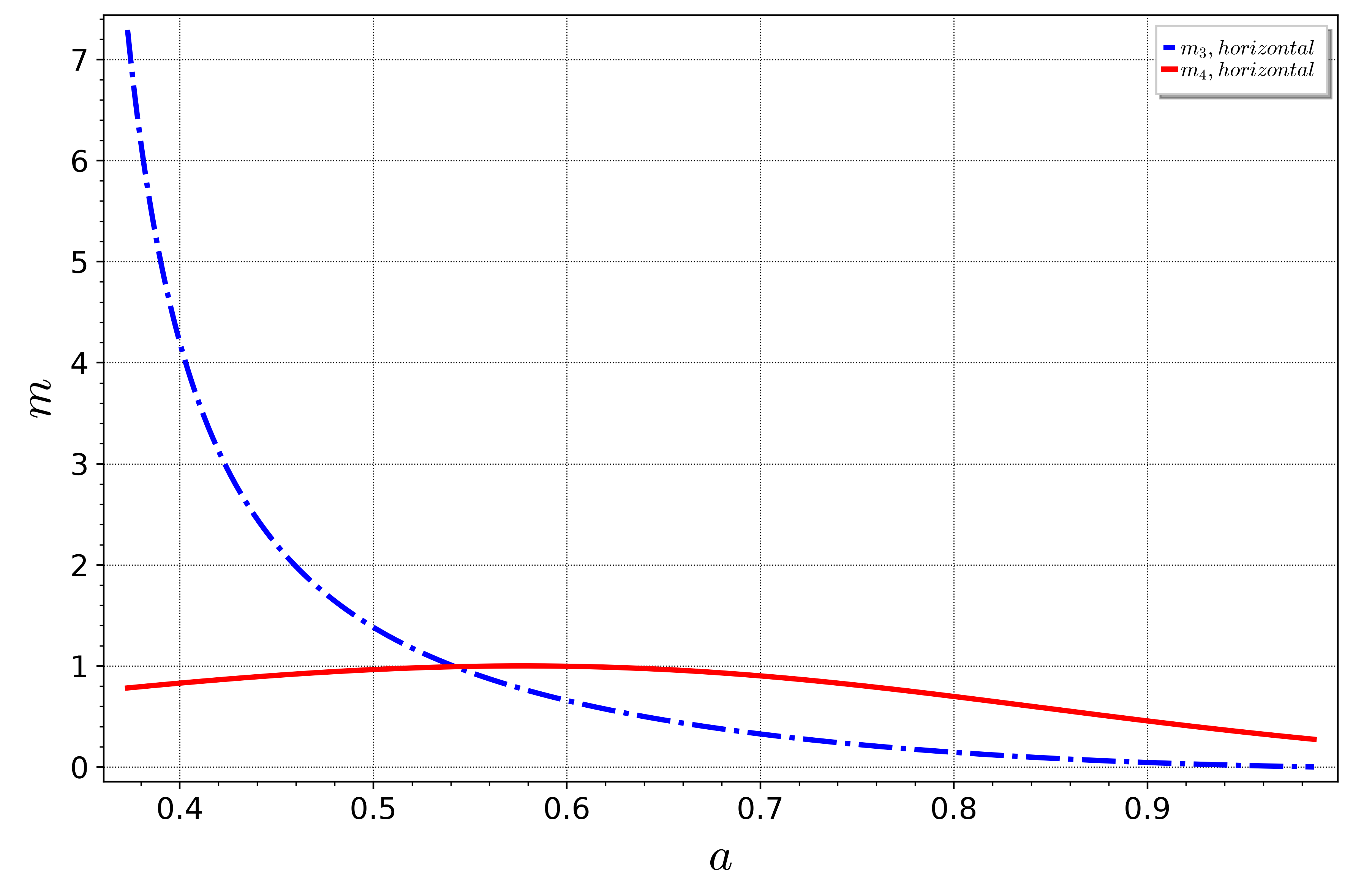}
    \caption{Variation of the masses $m_3$ (green dotted) and $m_4$ (pink dashed) along the vertical
degeneracy locus $\Gamma_v$ (left) and the masses $m_3$ (blue dash-dotted) and $m_4$ (red solid) along the horizontal degeneracy locus
$\Gamma_h$ (right). The intersections $m_3=m_4$ correspond to the two
degenerate configurations in the two-pairs-of-equal-masses family. In both cases, $m_3$ ranges from $0$ to $\infty.$}
    \label{fig:mass-degeneracy}
\end{figure}

\begin{remark}[\textbf{Variation of the masses along the degeneracy loci}]
\label{rem:mass-along-degeneracy}

Figure~\ref{fig:mass-degeneracy} illustrates the variation of the masses
$m_3$ and $m_4$ along the vertical and horizontal degeneracy loci. The two
branches exhibit markedly different mass behavior.

Along the vertical degeneracy locus
\[
\Gamma_v=\{(a,b):F_{J_{2v}}(a,b)=0\},
\]
the mass $m_3$ increases monotonically with $a$. Near the left endpoint of
the branch, $m_3$ approaches zero, whereas it grows rapidly to infinity as the right
endpoint is approached. In contrast, $m_4$ varies much more slowly: it
initially decreases, attains a value close to $1$, and then increases
slightly. Thus, a large range of mass ratios can produce degeneracy in the
symmetry-preserving vertical direction.

The behavior along the horizontal degeneracy locus
\[
\Gamma_h=\{(a,b):F_{J_{2h}}(a,b)=0\}
\]
is essentially complementary. As $a$ increases along this branch, $m_3$
decreases monotonically, from a very large value approaching infinity near the left
endpoint to a value approaching zero near the right endpoint. The mass
$m_4$, on the other hand, remains of order one: it first increases toward
$1$ and then decreases. Hence the symmetry-breaking degeneracy represented
by $\Gamma_h$ also persists over a substantial range of mass ratios, rather
than occurring only when the four masses are nearly equal.

An especially interesting feature of the two plots is the occurrence of
points at which $m_3=m_4.$  These intersections correspond to concave central configurations with two
pairs of equal masses,  $ m_1=m_2=1, m_3=m_4=\mu.$ 
On the horizontal degeneracy branch this gives the symmetry-breaking
degenerate configuration, whereas on the vertical degeneracy branch it gives the fold-type degenerate
configuration. These are precisely the two types of degeneracy  in Case 2. 

\begin{figure}[ht]
    \centering
    \includegraphics[width=0.6\textwidth,height=3in]{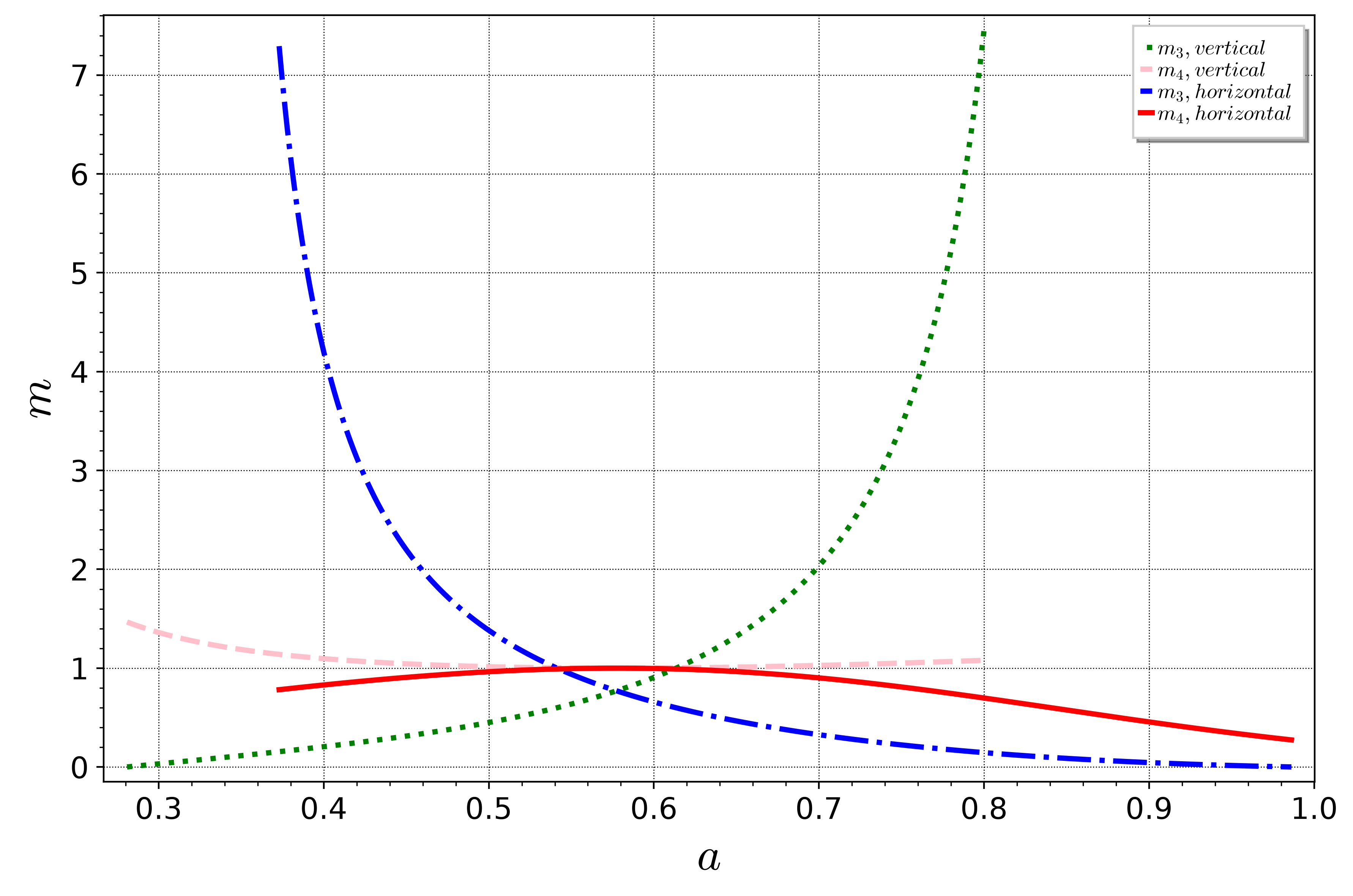}
    \caption{The masses $m_3$ and $m_4$ along the vertical and horizontal degeneracy curves, parametrized by $a$. The green dotted and blue dash-dotted curves represent $m_3$ along the vertical and horizontal degeneracy curves, respectively, while the pink dashed and red solid curves represent the corresponding values of $m_4$. The two $m_3$ curves intersect at the same parameter value at which the two $m_4$ curves intersect, namely at the special configuration $E=\left(\frac{\sqrt{3}}{3},\sqrt{3}\right) \approx (0.5773502693,1.732050808).$  $m_3(E)\approx 0.7704869546,m_4(E)=1.$
Thus, at $E$, the vertical and horizontal degeneracy conditions correspond to the same mass parameters.}
\label{fig:mass_deg_together}
\end{figure}

Another interesting feature of Figure~\ref{fig:mass_deg_together} {where we put the vertical and horizontal degeneracy curves together} is that the mass parameters associated with the two different types of degeneracy agree precisely at the special configuration $E$. More specifically, although the functions $m_3$ and $m_4$ exhibit quite different behavior along the horizontal and vertical degeneracy curves away from $E$, the two values of $m_3$ coincide at $E$, and the two values of $m_4$ also coincide there. Hence the horizontal and vertical degeneracy conditions not only occur at the same geometric configuration $(a,b)=E$, but also determine the same corresponding mass distribution. This provides further evidence that the coincidence of the two degeneracy curves at $E$ is intrinsic to the special symmetry of this configuration rather than an accidental intersection in the $(a,b)$-parameter space. In particular, at
$E=\left(\frac{\sqrt{3}}{3},\sqrt{3}\right),$ $m_3(E)\approx 0.7704869546,m_4(E)=1,$
and the bodies $q_1$, $q_2$, and $q_4$ form an equilateral triangle, while $q_3$ lies at its center. The simultaneous agreement of the mass parameters, together with the tangency of the horizontal and vertical degeneracy curves at $E$, highlights the exceptional role played by this highly symmetric configuration in the full-space degeneracy problem.

\begin{figure}[ht]
    \centering
    \includegraphics[width=0.6\textwidth,height=3in]{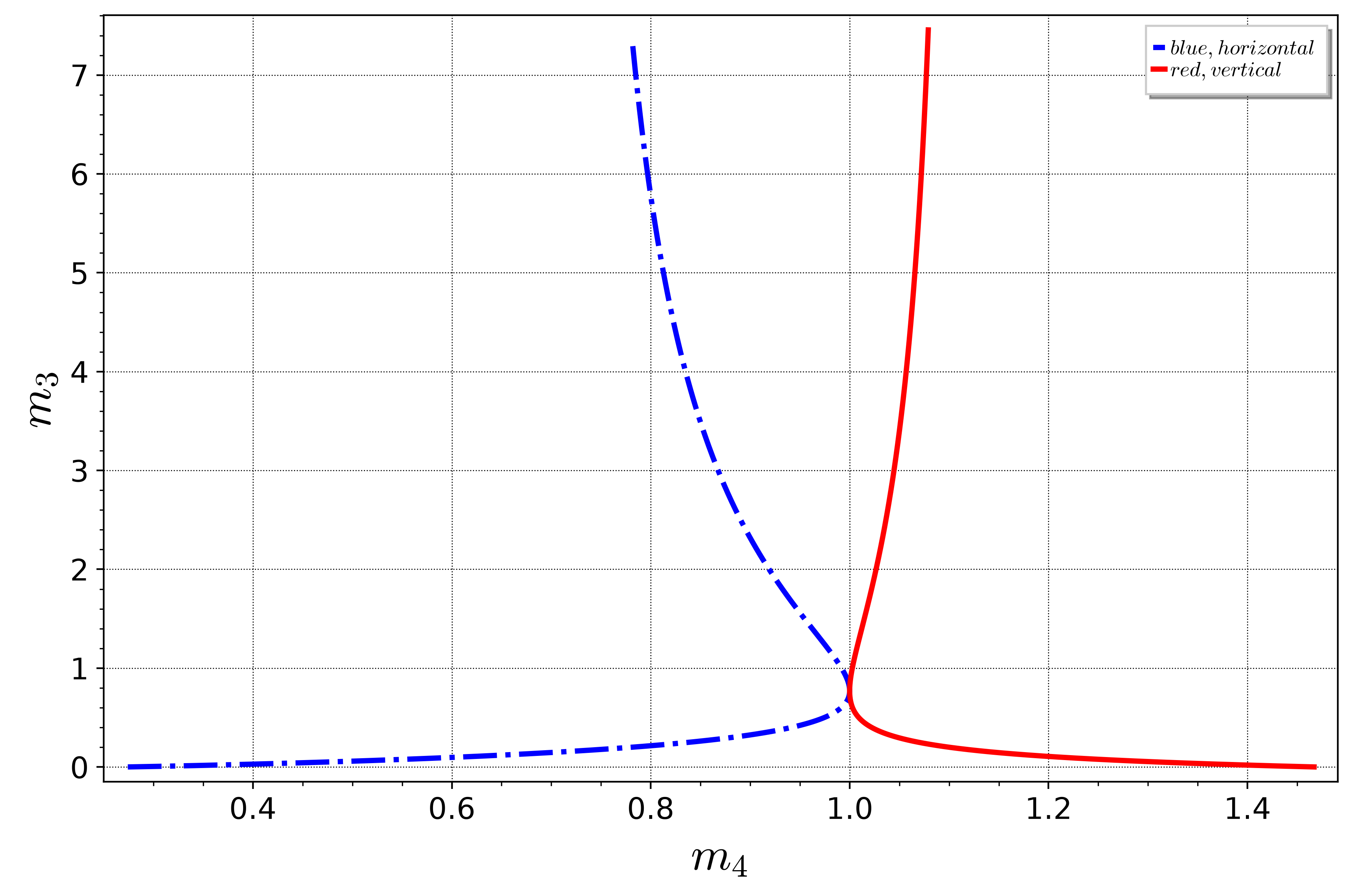}
   
\caption{ The mass relation $(m_4,m_3)$ along the horizontal and vertical degeneracy curves. The blue dash-dotted curve is the image of $\Gamma_h$ and the red solid curve is the image of $\Gamma_v$ in the $(m_4,m_3)$-plane. The two mass curves meet at the
configuration $E=(\sqrt{3}/3,\sqrt{3})$, corresponding to $m_4=1$ and
$m_3\approx0.7704869546$. Away from this common point, the horizontal and vertical
degeneracy branches exhibit distinctly different mass distributions.
}
\label{fig:m3_vs_m4_deg}
\end{figure}

Figure~\ref{fig:m3_vs_m4_deg} shows the mass relation $(m_4,m_3)$ along
the horizontal and vertical degeneracy curves $\Gamma_h$ and $\Gamma_v$.
Although the two degeneracy branches have quite different behavior in the
mass parameter space, their images meet at the mass pair corresponding to
the special configuration
\[
E=\left(\frac{\sqrt3}{3},\sqrt3\right),
\qquad
(m_4,m_3)\approx(1,0.7704869546).
\]
As one moves away from $E$, the two branches separate rapidly and exhibit
distinct mass distributions. Along the horizontal branch $\Gamma_h$ (blue),
$m_3$ becomes large as $m_4$ decreases, whereas along the vertical branch
$\Gamma_v$ (red), $m_3$ becomes large as $m_4$ increases. Thus, the figure
provides a complementary view of the horizontal and vertical degeneracies
in the $(m_4,m_3)$-parameter space and further highlights the distinguished
role of $E$, where the two types of degeneracy coincide.
\end{remark}

\section{Appendix: Proof of Lemma \ref{lem:cancellation-normalization} }
Here we give the proof of Lemma \ref{lem:cancellation-normalization} which is used to prove Theorem \ref{Prop:J2Entries}.

\begin{proof}
{We start with a reminding that all vectors in the proof should be understood as column vectors.} Set
\[
C=B_2B_1^{-1}
=
\begin{pmatrix}
\frac12&\frac a2&\frac12&-\frac a2\\
-\frac a2&\frac12&\frac a2&\frac12\\
\frac12&\frac b2&\frac12&-\frac b2\\
-\frac b2&\frac12&\frac b2&\frac12
\end{pmatrix}.
\]
Since \(m_1=m_2=1\) and the configuration is symmetric with respect
to the \(y\)-axis, its center of mass is
\[
c=(0,c_y),
\qquad
c_y=\frac{m_3a+m_4b}{2+m_3+m_4}.
\]

Introduce the stacked vectors
\[
s_1=
\begin{pmatrix}
q_1-c\\
q_2-c
\end{pmatrix}
=
\begin{pmatrix}
-1\\
-c_y\\
1\\
-c_y
\end{pmatrix},
\qquad
s_2=
\begin{pmatrix}
q_3-c\\
q_4-c
\end{pmatrix}
=
\begin{pmatrix}
0\\
a-c_y\\
0\\
b-c_y
\end{pmatrix}.
\]
A direct multiplication gives
\begin{equation}\label{eq:Cs1s2}
Cs_1=s_2.
\end{equation}

Let
\[
R=
\begin{pmatrix}
1&0\\
0&1\\
1&0\\
0&1
\end{pmatrix}.
\]
Thus, for any \(v\in\mathbb{R}^2\), the vector \(Rv\) consists of
two copies of \(v\). Another direct multiplication gives
\begin{equation}\label{eq:CRR}
CR=R.
\end{equation}
This identity reflects the fact that the translation directions are
among the four columns used to construct \(P\).

Write
\[
q^{(1)}=(x_1,y_1,x_2,y_2),
\qquad
q^{(2)}=(x_3,y_3,x_4,y_4),
\]
and decompose
\[
Dc=
\begin{pmatrix}
D_1c&D_2c
\end{pmatrix},
\qquad
D\lambda=
\begin{pmatrix}
D_1\lambda&D_2\lambda
\end{pmatrix},
\]
where
\[
D_kc=\frac{\partial c}{\partial q^{(k)}},
\qquad
D_k\lambda=\frac{\partial\lambda}{\partial q^{(k)}},
\qquad k=1,2.
\]

For each \(i\),
\[
D\bigl(\lambda(q)(q_i-c(q))\bigr)
=
\lambda\,D(q_i-c)
+(q_i-c)\,D\lambda.
\]
Consequently, the blocks of \(D\mathcal{N}\) involving derivatives
with respect to \(q^{(2)}\) are
\begin{align}
N_{12}
&=
-\lambda R D_2c+s_1D_2\lambda,
\label{eq:N12normalization}\\
N_{22}
&=
\lambda\bigl(I_4-RD_2c\bigr)
+s_2D_2\lambda.
\label{eq:N22normalization}
\end{align}
Here \(s_kD_2\lambda\) denotes the $4\times4$ matrix of the product of the column vector
\(s_k\) and the row vector \(D_2\lambda\).

Using \eqref{eq:N12normalization} and
\eqref{eq:N22normalization}, we obtain
\begin{align*}
N_{22}-CN_{12}
&=
\lambda\bigl(I_4-RD_2c\bigr)+s_2D_2\lambda
-C\bigl(-\lambda RD_2c+s_1D_2\lambda\bigr)\\
&=
\lambda I_4
-\lambda RD_2c
+\lambda CRD_2c
+\bigl(s_2-Cs_1\bigr)D_2\lambda.
\end{align*}
By \eqref{eq:Cs1s2} and \eqref{eq:CRR},
\[
s_2-Cs_1=0,
\qquad
CR-R=0.
\]
It follows that
\[
N_{22}-CN_{12}
=
\lambda I_4.
\]
\end{proof}



{\bf Acknowledgements:}
Shanzhong Sun is partially supported by NSFC (Nos. 12171327, 12261131498). Zhifu Xie is partially supported by the Wright W. and Annie Rea Cross Endowment Funds at the University of Southern Mississippi and by the National Science Foundation (Award Number (FAIN): 2447675).


\end{document}